\documentclass{amsart}
 
\usepackage[latin1]{inputenc}
\usepackage[english]{babel}
\usepackage{amssymb,amsmath,amsthm}
\usepackage{amsfonts}
\usepackage{rotating}
\usepackage{url}
\input{xy}\xyoption{curve}

\theoremstyle{plain}
\newtheorem{theorem}{Theorem}

\newtheorem*{main}{Associativity Lemma}
\newtheorem{lemma}{Lemma}
\newtheorem{proposition}{Proposition}
\newtheorem{fact}{Fact}

\newcommand{\interval}[2]{\ensuremath[{#1},{#2}]}

\begin{document}
\title[Lattice of Equational Classes of Boolean Functions]
{On the Lattice of Equational Classes of Boolean Functions and Its Closed Intervals}

\author{Miguel Couceiro}
\address{Department of Mathematics, Statistics and Philosophy\\
University of Tampere\\ 
Kanslerinrine 1, 33014 Tampere, Finland}
\email{Miguel.Couceiro@uta.fi}
\thanks{The work of the author was partially supported by the Graduate School in Mathematical
 Logic MALJA, and by grant \#28139 from the Academy of Finland}


\begin{abstract} Let $A$ be a finite set with $\lvert A\rvert\geq 2$. 
The composition of two classes $\mathcal{I}$ and $\mathcal{J}$ of 
operations on $A$, is defined as the set of all composites 
$f(g_1, \ldots, g_n)$ with $f \in \mathcal{I}$ and $g_1, \ldots, g_n \in \mathcal{J}$. This
binary operation gives a monoid structure to the set ${\bf E}_A$ 
of all equational classes of operations on $A$. 

The set ${\bf E}_A$ of equational classes of operations on $A$ also constitutes a complete distributive 
lattice under intersection and union. Clones of operations, i.e. classes containing all projections
 and idempotent under class composition, also form a lattice which is strictly contained in ${\bf E}_A$. 
In the Boolean case $\lvert A\rvert= 2$, the lattice
${\bf E}_A$ contains uncountably many ($2^{\aleph _0}$) equational classes, but only 
countably many of them are clones. 

 The aim of this paper is to provide a better understanding of this uncountable
 lattice of equational classes of Boolean functions, 
by analyzing its ``closed" intervals $\interval{\mathcal{C}_1}{\mathcal{C}_2}$, 
for idempotent classes ${\mathcal{C}_1}$ and ${\mathcal{C}_2}$.
For $\lvert A\rvert= 2$, we give a complete classification of all closed
 intervals $\interval{\mathcal{C}_1}{\mathcal{C}_2}$ in terms of their size,
 and provide a simple, necessary and sufficient condition characterizing the uncountable closed intervals of ${\bf E}_A$.       
\end{abstract}

\maketitle

\section{Introduction}

The characterization of the classes of operations on a set $A$,
 definable by means of functional equations, was first obtained in the Boolean case $A=\{0,1\}$
by Ekin, Foldes, Hammer and Hellerstein in \cite{EFHH}, and in a different framework by Pippenger \cite{Pi2}.
This result was extended in \cite{CF2} to arbitrary non-empty underlying sets $A$, where it was shown that these equational
classes are essentially those classes $\mathcal{K}$ satisfying $\mathcal{K}\mathcal{P}_ A=\mathcal{K}$, 
where $\mathcal{P}_ A$ denotes the class containing only projections on $A$. 
From this characterization it follows that the set ${\bf E}_A$ of equational classes on $A$
 constitutes a complete distributive lattice which properly contains the set of all clones on $A$.
In fact, the classification of operations into equational classes is 
much finer than the classification into clones. For example, in the Boolean case $\lvert A\rvert =2$, 
there are uncountably many equational classes on $A$ (see e.g. \cite{Pi2}), but only 
countably many of them are clones (see \cite{Post}). 

Thus it seems very hard to achieve a complete description of the lattice 
${\bf E}_A$, even in the case $\lvert A\rvert =2$.
Nevertheless, the subset of those equational classes which are idempotent under class composition, 
induces a subdivision of ${\bf E}_A$ into sublattices 
$\interval{\mathcal{C}_1}{\mathcal{C}_2}$ which are in addition closed under class composition.

In this paper, we study the closed intervals of ${\bf E}_A$. The distribution of the equational classes
 into these intervals is not uniform: some intervals are countable, while others are uncountable. Thus 
it is natural to ask which are the uncountable intervals of ${\bf E}_A$.

We answer this question for $A=\mathbb{B}=\{0,1\}$.
In the next section, we provide definitions and terminology as well as some preliminary results used
in the sequel. In Section 3, we introduce the lattice ${\bf E}_{A} $ of equational classes of operations on $A$,
and present some facts and general results concerning this lattice and its 
 intervals $\interval{\mathcal{C}_1}{\mathcal{C}_2}$, 
for idempotent classes ${\mathcal{C}_1}$ and ${\mathcal{C}_2}$.
In particular, we verify that an interval 
$\interval{\mathcal{C}_1}{\mathcal{C}_2}$ is uncountable if and only if there is an infinite antichain
of operations in $\mathcal{C}_2\setminus {\mathcal{C}_1}$ with respect to the 
pre-order $\preceq _{\bf V}$ defined on Section 2.

In Section 4, we focus on the lattice ${\bf E}_{\mathbb{B}} $
of equational classes of Boolean functions.
In view of the above characterization, we determine which intervals of ${\bf E}_{\mathbb{B}} $
contain only finite antichains, and provide infinite antichains of Boolean functions 
for the remaining closed intervals (Subsection 4.2). The classification 
of the closed intervals of ${\bf E}_{\mathbb{B}} $
in terms of size, is then presented in Subsection 4.3.
Using this classification, we derive in Subsection 4.4 a simpler, 
necessary and sufficient condition characterizing the uncountable closed intervals of ${\bf E}_{\mathbb{B}} $.  
 
\section{Basic notions and preliminary results}

Throughout the paper, let $A$ be a finite set with $\lvert A\rvert\geq 2$. 
An \emph{operation on $A$} is a map
$f: A^n \rightarrow A$, where $n$ is a positive integer called the \emph{arity} of $f$.
If $A=\mathbb{B}=\{0,1\}$, then $f$ is called a \emph{Boolean function}. By a \emph{class on $A$} we simply mean a subset  $\mathcal{I}\subseteq \underset{n \geq 1}{\bigcup} A^{A^n}$.

The \emph{essential arity} of an $n$-ary operation $f: A^n \rightarrow A$ 
is the cardinality of the set of indices 
 \begin{displaymath}
\begin{array}{l}
I=\{1\leq i\leq n: \textrm{ there are }a_1,\ldots ,a_i,b_i,a_{i+1},\ldots ,a_n\textrm{ with }a_i\not=b_i
\textrm{ and }\\
f(a_1,\ldots ,a_{i-1},a_i,a_{i+1},\ldots ,a_n)\not=f(a_1,\ldots ,a_{i-1},b_i,a_{i+1},\ldots ,a_n)\}
\end{array}
\end{displaymath}
For each $i\in I$, the $i$th variable of $f$ is said to be \emph{essential}. A variable $x_i$ of $f$
is called \emph{dummy} if $i\not\in I$. By definition it follows that
constant operations have only dummy variables. Operations of essential arity at most 1
 are usually called \emph{quasi-monadic}. An operation $f: A^n \rightarrow A$ is said to be 
\emph{idempotent}, if $f(x_1,\ldots ,x_1)=x_1$.  

For any maps $g_1,\ldots ,g_n:A^m\rightarrow A$ and $f: A^n \rightarrow A$, 
 their \emph{composition} is defined as the map 
$f(g_1,\ldots ,g_n):A^m\rightarrow A$ given by 
$f(g_1,\ldots ,g_n)({\bf a})=f(g_1({\bf a}),\ldots ,g_n({\bf a}))$, for every $ {\bf a}\in A^m$.
An $n$-ary operation $f:A^n\rightarrow A$ is said to be \emph{associative} if for any $2n-1$-ary projections
$p_k$, $1\leq k\leq 2n-1$, and every pair of indices $1\leq i<j\leq n$, we have 
\begin{displaymath}
\begin{array}{l}
f(p_1,\ldots ,p_{i-1},f(p_i,\ldots ,p_{i+n-1}),p_{i+n},\ldots ,p_{ 2n-1})=
\\ =f(p_1,\ldots ,p_{j-1},f(p_j,\ldots ,p_{j+n-1}),p_{j+n},\ldots ,p_{ 2n-1})
\end{array}
\end{displaymath}
Note that if $f$ is associative of essential arity $n\geq 2$, then it does not have dummy variables.
Also, by definition it follows that each member of the family $(f^k)_{k\geq 0}$ given by the recursion
\begin{enumerate}
\item $f^0=f(x_1,\ldots ,x_1)$ and $f^1=f$,
\item $f^k=f^{k-1}(x_1,\ldots ,x_{(k-1)(n-1)},f(x_{(k-1)(n-1)+1},\ldots ,x_{k(n-1)+1}))$ 
\end{enumerate}
is also associative. This notion of associativity for $n$-ary operations plays a fundamental role  
in the generalization of groups to $n$-groups (polyadic goups). For an early reference 
see e.g. \cite{Post2}, and for a bibliographic survey see \cite{G}. 

If $\mathcal{I},\mathcal{J}\subseteq \underset{n \geq 1}{\bigcup}  A^{A^n}$, then
 the \emph{class composition} $\mathcal{I}\mathcal{J}$ is defined as the set
 \begin{displaymath}
 \mathcal{I}\mathcal{J}=\{f(g_1,\ldots ,g_n)\mid n,m\geq 1, f\textrm{ $n$-ary in $\mathcal{I}$, }g_1,\ldots ,g_n\textrm{ $m$-ary in $\mathcal{J}$} \}. 
\end{displaymath}
If $\mathcal{I}$ is a singleton, $\mathcal{I}=\{f\}$, then we write $f\mathcal{J}$ instead of $\{f\}\mathcal{J}$.   
Note that class composition is monotone, i.e. if $\mathcal{I}_1\subseteq \mathcal{I}_2$ and
 $\mathcal{J}_1\subseteq \mathcal{J}_2$, then 
$\mathcal{I}_1\mathcal{J}_1\subseteq \mathcal{I}_2\mathcal{J}_2$. 
 
Let $\mathcal{P}_ A$ denote the class containing only projections on $A$.
An $m$-ary operation $g$ on $A$ 
is said to be obtained from an $n$-ary operation $f$ on $A$
 by \emph{simple variable substitution}, denoted $g\preceq _{\bf V}f$,
if there are $m$-ary projections $p_1,\ldots ,p_n\in \mathcal{P}_ A$ such that 
$g=f(p_1,\ldots ,p_n)$. In other words,  
  \begin{displaymath}
g\preceq _{\bf V} f \quad \text{if and only if }\quad  g\mathcal{P}_ {A}\subseteq f\mathcal{P}_ {A}.
\end{displaymath}
Thus $\preceq _{\bf V}$ constitutes a pre-order (reflexive and transitive)
on $\underset{n \geq 1}{\bigcup}  A^{A^n}$. 
If $g\preceq _{\bf V}f$ and $f\preceq _{\bf V}g$, then $g$ and $f$ are said to be 
\emph{equivalent}. Note that if $g\preceq _{\bf V}f$ but $f\not\preceq _{\bf V}g$, then the essential arity of $g$ is less
than the essential arity of $f$, and hence, every descending chain with 
respect to $\preceq _{\bf V}$ must be finite.

If $g\not\preceq _{\bf V}f$ and $f\not\preceq _{\bf V}g$, then $g$ and $f$ 
are said to be \emph{incomparable}. By an \emph{antichain} of operations we simply mean a set of pairwise
incomparable operations with respect to $\preceq _{\bf V}$. 

We say that an operation $g$ is \emph{quasi-associative} if there is 
an associative operation $f$ such that $g\preceq _{\bf V}f$. Clearly, every associative operation
is also quasi-associative, but there are quasi-associative operations which are not associative.
By definition we have:

\begin{proposition}\label{quasi} Let $f$ be an associative operation. If $g\preceq _{\bf V}f$ is not associative,
 then it is obtained from $f$ by addition of inessential variables. 
\end{proposition} 
  We refer to operations which are not quasi-associative as \emph{non-associative}.  

 A class $\mathcal{K}\subseteq \underset{n \geq 1}{\bigcup} A^{A^n}$ of operations on $A$,
is said to be \emph{closed under simple variable substitutions}
if each operation obtained from a operation $f$ in $\mathcal{K}$ by simple variable substitution
 is also in $\mathcal{K}$. In other words, the class $\mathcal{K}$
 is closed under simple variable substitutions if and only
if $ \mathcal{K}\mathcal{P}_ A =\cup _{f\in \mathcal{K}}f\mathcal{P}_ A \subseteq  \mathcal{K}$. Clearly, this condition
is equivalent to $\mathcal{K}\mathcal{P}_ A =  \mathcal{K}$.
We denote by ${\bf V}_{A} $ the set of all classes of operations on $A$ 
closed under simple variable substitutions. 

\bigskip

Recall that a \emph{monoid} with universe $M$ is an algebraic structure $\langle M,\cdot  \rangle$ with an associative
 operation $\cdot :M^2\rightarrow M$, and an identity element, usually denoted by $1_M$. In other words,
a monoid is a semigroup with an identity element.
 If $\leq $ is a partial order on $M$, and if for every $x,y,z,w\in M$ the following condition holds
 \begin{displaymath}
\text{if}\quad x\leq y, \quad \text{then }\quad  z\cdot x\cdot w\leq z\cdot y\cdot w, 
 \end{displaymath}
then $\langle M,\cdot  \rangle$ is called a \emph{partially ordered monoid}. 

\begin{theorem}
The set $ {\bf V}_{A} $ constitutes a partially ordered monoid with respect to class composition, with 
 $\mathcal{P}_ A$ as its identity. 
\end{theorem}

To prove Theorem 2, we need the following

\begin{main}\emph{(In \cite{CF1,CF2}:)}
Let $A$ be a finite set with $\lvert A\rvert\geq 2$, and let $\mathcal{I}$, $\mathcal{J}$, and   
$\mathcal{K}$ be classes of operations on $A$. The following hold:
\begin{itemize}
\item[(i)] $(\mathcal{I}\mathcal{J})\mathcal{K}\subseteq \mathcal{I}(\mathcal{J}\mathcal{K})$;
\item[(ii)] If $\mathcal{J}$ is closed under simple variable substitutions,
then $(\mathcal{I}\mathcal{J})\mathcal{K}=\mathcal{I}(\mathcal{J}\mathcal{K})$.
\end{itemize}
\end{main}

\begin{proof}[Proof of Theorem 1.]
By the characterization of the equational classes given in Theorem 1, 
 and using the Associativity Lemma, it follows that 
class composition is associative on ${\bf V}_{A}$. Clearly, for every
$\mathcal{K}\in {\bf V}_{A} $, $\mathcal{P}_ A\mathcal{K}=\mathcal{K}$ 
and $\mathcal{K}\mathcal{P}_ A=\mathcal{K}$. Since the members of
${\bf V}_{A}$ are closed under variable substitutions, again by making use of the Associativity Lemma it follows
that $(\mathcal{K}_1\mathcal{K}_2)\mathcal{P}_ A=\mathcal{K}_1(\mathcal{K}_2\mathcal{P}_ A)=\mathcal{K}_1\mathcal{K}_2$.
Furthermore, class composition is order-preserving,
 and the proof of Theorem 2 is complete. 
\end{proof}

An \emph{idempotent} of a monoid $M$ is an element $e$ of M such that
$e\cdot e = e$.

\begin{fact}
The idempotents of $ {\bf V}_{A} $ containing $\mathcal{P}_ A$ are exactly the \emph{clones} on $A$.
 Moreover, $\mathcal{P}_ A$ is the smallest clone on $A$ and each clone is closed 
under simple variable substitutions.
\end{fact}

\begin{proposition}
 If $\mathcal{C}_1,\mathcal{C}_2\in {\bf V}_{A} $ are idempotents such that $\mathcal{C}_1\subseteq \mathcal{C}_2$, then 
 \begin{displaymath}
\interval{\mathcal{C}_1}{\mathcal{C}_2} =\{ \mathcal{K}\in {\bf V}_{A}: 
\mathcal{C}_1\subseteq \mathcal{K}\subseteq \mathcal{C}_2\}
 \end{displaymath}
is a semigroup.
\end{proposition}

\begin{proof}
The proof of Proposition 1 follows from the fact that if 
$\mathcal{C}_1\subseteq \mathcal{K}_1,\mathcal{K}_2\subseteq \mathcal{C}_2$,
for idempotents $\mathcal{C}_1\subseteq \mathcal{C}_2$,
then $\mathcal{C}_1\subseteq \mathcal{K}_1\mathcal{K}_2\subseteq \mathcal{C}_2$.  
\end{proof}
Note that not all intervals $\interval{\mathcal{K}_1}{\mathcal{K}_2} $, for arbitrary 
${\mathcal{K}_1}, {\mathcal{K}_2}\in {\bf V}_A$ are closed under class composition.  
We refer to the sets $\interval{\mathcal{C}_1}{\mathcal{C}_2} $, for idempotents
${\mathcal{C}_1}$ and ${\mathcal{C}_2}$, as 
\emph{closed intervals}. If $\mathcal{C}_1$ 
is covered by $\mathcal{C}_2$, i.e., if for every idempotent $\mathcal{C}$ such that
 $\mathcal{C}_1\subseteq \mathcal{C}\subseteq \mathcal{C}_2$ we have $\mathcal{C}=\mathcal{C}_1$ or 
$\mathcal{C}=\mathcal{C}_2$, then we say that the interval 
$\interval{\mathcal{C}_1}{\mathcal{C}_2}$ is \emph{minimal}. 

\section{The lattice of equational classes of operations on $A$}

A \emph{functional equation} (for operations on $A$) is a formal expression
 \begin{displaymath}
\begin{array}{l}
h_1( {\bf f} (g_1({\bf v}_1,\ldots ,{\bf v}_p)),\ldots ,{\bf f} (g_m({\bf v}_1,\ldots ,{\bf v}_p)))=
\\=h_2( {\bf f} ({g'}_1({\bf v}_1,\ldots ,{\bf v}_p)),\ldots ,{\bf f} ({g'}_t({\bf v}_1,\ldots ,{\bf v}_p))) \qquad (1)
\end{array}
\end{displaymath}
where $m,t,p\geq 1$, $h_1: A^m \rightarrow A$, $h_2: A^t \rightarrow A$, each $g_i$ and ${g'}_j$ is a map
$A^p\rightarrow A$, the ${\bf v}_1,\ldots ,{\bf v}_p$ are $p$ distinct symbols called \emph{vector variables}, and 
$\bf f$ is a distinct symbol called \emph{function symbol}. 

For $n\geq 1$, we denote by $\bf n$ the set ${\bf n}=\{1,\ldots ,n\}$, so that an $n$-vector ($n$-tuple) $v$ in $A^n$
 is a map $v:{\bf n}\rightarrow A$.
For an $n$-ary operation on $A$, $f: A^n \rightarrow A$, we say that $f$ \emph{satisfies} 
the equation $(1)$ if, for all ${v}_1,\ldots ,{v}_p\in A^n$, we have
 \begin{displaymath}
\begin{array}{l}
h_1( {f} (g_1({v}_1,\ldots ,{v}_p)),\ldots ,{f} (g_m({v}_1,\ldots ,{ v}_p)))=
\\=h_2( {f} ({g'}_1({v}_1,\ldots ,{v}_p)),\ldots ,{f} ({g'}_t({v}_1,\ldots ,{v}_p))) 
\end{array}
\end{displaymath}
 A class $\mathcal{K}$ of operations on $A$ is said to be \emph{defined},
 or \emph{definable}, by a set $\mathcal{E}$ of functional equations, if $\mathcal{K}$ is the class of all 
those operations which satisfy every member of $\mathcal{E}$. We say that a class $\mathcal{K}$ is 
\emph{equational} if it is definable by some set of functional equations.  
We denote by ${\bf E}_{A} $ the set of all equational classes of operations on $A$. 
The following result was first obtained by Ekin, Foldes, Hammer and Hellerstein \cite{EFHH} 
for the Boolean case 
$A=\mathbb{B}=\{0,1\}$.

\begin{theorem}\label{theorem:Equational}\emph{(In \cite{CF2}:)} The equational classes of 
operations on $A$ are exactly those classes
 that are closed under simple variable substitutions. 
\end{theorem}
In other words, the sets ${\bf E}_{A} $ and ${\bf V}_{A} $ are exactly the same. 
By definition of class composition, it follows that 
  \begin{displaymath}
\begin{array}{l}
(\mathcal{K}_1\cup \mathcal{K}_2)\mathcal{P}_ A=
\mathcal{K}_1\mathcal{P}_ A\cup \mathcal{K}_2\mathcal{P}_ A\textrm{  and  }  
 (\mathcal{K}_1\cap \mathcal{K}_2)\mathcal{P}_ A=
\mathcal{K}_1\mathcal{P}_ A\cap \mathcal{K}_2\mathcal{P}_ A  
\end{array}
\end{displaymath}for every 
$\mathcal{K}_1,\mathcal{K}_2\subseteq \underset{n \geq 1}{\bigcup} A^{A^n}$. From these facts and 
using Theorem~\ref{theorem:Equational}, we obtain: 

\begin{fact}  
The set ${\bf E}_{A} $ of all equational classes of 
operations on $A$ constitutes a complete distributive lattice under intersection and union, 
with $\emptyset $ and $\underset{n \geq 1}{\bigcup} A^{A^n}$ as minimal and maximal elements, respectively. 
 \end{fact}

The set ${\bf E}_{A} $ constitutes a closure system, 
and thus each equational class can be described by a set of ``generators".
In fact, by making use of Theorem~\ref{theorem:Equational}, we see that the smallest equational class on $A$ containing a set  
$\mathcal{K}\subseteq \underset{n \geq 1}{\bigcup} A^{A^n}$ is the class composition $\mathcal{K}\mathcal{P}_ A$.
The equational class $\mathcal{K}\mathcal{P}_ A$ is said to be \emph{generated} by $\mathcal{K}$. If 
$\mathcal{K}$ is a finite set of operations, then we say that 
$\mathcal{K}\mathcal{P}_ A$ is \emph{finitely generated}.
 \begin{theorem}
\label{theorem:monadic} Let $A$ be a finite set, and let $\mathcal{C}$ be an idempotent of ${\bf E}_{A} $. 
Then $\mathcal{C}$ is a finitely generated equational class if and only if $\mathcal{C}$ contains
 only quasi-monadic operations. Furthermore, only finitely many equational classes in ${\bf E}_A$
 are finitely generated. \end{theorem}

\begin{proof}
Note that for each finite $A$, there are only finitely many quasi-monadic operations (up to equivalence),
 and thus the equational classes containing only quasi-monadic operations 
must be finitely generated.
In particular, the equational classes on $A$ which are idempotent and containing only quasi-monadic 
operations are finitely generated.

To see that these are indeed the only equational classes on $A$ which are idempotent 
and finitely generated, let $\mathcal{C}$ be an idempotent equational class containing an operation $f$ 
of essential arity $n>1$. Now, if $\mathcal{C}$ were finitely generated, then 
there would be an integer $N\geq n$, and an $N$-ary generator $f_N$ of essential arity $N$, 
such that every operation in $\mathcal{C}$ has essential arity 
at most $N$. But the $2N-1$-ary operation 
 \begin{displaymath}
 f'_N(x_1,\ldots ,x_{2N-1})= f_N(x_1,\ldots ,x_{N-1},f_N(x_{N},\ldots ,x_{2N-1}))
\end{displaymath}
has essential arity equal to $2N-1$ and since $\mathcal{C}$ is idempotent, it must be in $\mathcal{C}$,
which constitutes a contradiction. Thus indeed $\mathcal{C}$ cannot be finitely generated.

The last claim follows from the fact that there are only finitely many pairwise incomparable
 quasi-monadic operations on a finite set.
\end{proof}

By reasoning as in the proof of Theorem~\ref{theorem:monadic},
 it is not difficult to verify that the following also holds:

   \begin{theorem}
\label{theorem:finite intervals} Let $A$ be a finite set, and let 
$\mathcal{C}_1,\mathcal{C}_2\in {\bf E}_A$ be two idempotent classes such that 
$\mathcal{C}_1\subseteq \mathcal{C}_2$. Then the interval
$\interval{\mathcal{C}_1}{\mathcal{C}_2} \subseteq {\bf E}_A$ is finite 
if and only if $\mathcal{C}_2\setminus \mathcal{C}_1$ contains only quasi-monadic operations. 
 \end{theorem}
 
The following theorem provides a necessary and sufficient for a closed interval to contain 
uncountably many equational classes. 

\begin{theorem}
\label{theorem:characterization}
Let $\mathcal{C}_1$ and $\mathcal{C}_2$ be two idempotent classes such that $\mathcal{C}_1\subseteq \mathcal{C}_2$.
Then there are uncountably many ($2^{\aleph _0}$) equational classes in 
$\interval{\mathcal{C}_1}{\mathcal{C}_2} $
if and only if $\mathcal{C}_2\setminus \mathcal{C}_1$ contains an infinite (countable) antichain of operations.
\end{theorem}

\begin{proof} Note that the set of all subsets of an infinite (countable) set is uncountable.
Also, distinct subsets of pairwise incomparable functions generate distinct equational classes and thus,
 if $F=(f_i)_{i\in I}$ is an infinite antichain operations in $\mathcal{C}_2\setminus \mathcal{C}_1$, then  
\begin{displaymath}
E= \{S\mathcal{P}_ A\cup \mathcal{C}_1: S\subseteq F\}
\end{displaymath}
is an uncountable ($2^{\aleph _0}$) set of equational classes 
in $\interval{\mathcal{C}_1}{\mathcal{C}_2} $.  

To see that the converse also holds, observe first that 
for each equational class $\mathcal{K} \in \interval{\mathcal{C}_1}{\mathcal{C}_2}$, 
the relative complement ${\mathcal{K}}^{\mathcal{C}_2}_{\mathcal{C}_1}$ of $\mathcal{K} $ in 
$\interval{\mathcal{C}_1}{\mathcal{C}_2}$, given by
  \begin{displaymath}
 \begin{array}{lll}  
 {\mathcal{K}}^{\mathcal{C}_2}_{\mathcal{C}_1}={ \mathcal{C}_1}\cup 
[(\underset{n \geq 1}{\bigcup} A^{A^n}\setminus \mathcal{K})\cap \mathcal{C}_2]
\end{array}
 \end{displaymath}
is completely determined by maximal antichains of its minimal (under $\preceq _{\bf V}$) operations,
 because there are no infinite descending chains with respect to $\preceq _{\bf V}$. 

Now suppose that every antichain in $\mathcal{C}_2\setminus \mathcal{C}_1$
is finite. Then it follows from the above observation that there are only countably many
relative complements of equational classes in 
$\interval{\mathcal{C}_1}{\mathcal{C}_2}$, and thus there are only countably many
equational classes in $\interval{\mathcal{C}_1}{\mathcal{C}_2}$, and the proof of the theorem
is complete.
\end{proof}

\section{The closed intervals of the lattice of equational classes of Boolean functions}

\subsection{Preliminaries} 
We denote by $\Omega =\underset{n \geq 1}{\bigcup}  \mathbb{B}^{\mathbb{B}^n}$
the set of all Boolean functions. 
The set $\mathbb{B}^n$ is a Boolean lattice (distributive and complemented) of $2^n$ elements
 under the component-wise order of vectors
 \begin{displaymath}
(a_1, \ldots, a_n) \preceq (b_1, \ldots, b_n) \textrm{ if and only if } a_i\leq b_i,\textrm{ for all }1 \leq i \leq n.
\end{displaymath}
In this way, all operations on the Boolean lattice $\mathbb{B}$ are generalized to $\mathbb{B}^n$ by
means of component-wise definitions. For example,
 the \emph{complement} of a vector 
$ {\bf a } = (a_1,\ldots , a_n)$ is also defined component-wise by
 $\bar {\bf a}  = (1 - a_1, \ldots , 1 - a_n)$.
 We denote the \emph{all-zero-vector} and the \emph{all-one-vector} by ${\bf 0} = (0, \ldots , 0)$ and
 ${\bf 1} = (1, \ldots , 1)$, respectively. 
The set $\mathbb{B}^{\mathbb{B}^n}$ is also a Boolean lattice of $2^{2^n}$ elements under the 
point-wise ordering of functions, i.e. 
 \begin{displaymath}
f\leq g \textrm{ if and only if } f({\bf a })\leq g({\bf a }),\textrm{ for all }{\bf a } \in \mathbb{B}^n.
\end{displaymath}
The functions (of any arity) having constant value $0$ and $1$ are denoted by 
${\bf 0}$ and ${\bf 1}$, respectively. 
The \emph{complement} of an $n$-ary Boolean function $f$ is the function 
$\bar f$ defined by  
${\bar f}({\bf a})=1-f({\bf a})$, for all ${\bf a } \in \mathbb{B}^n$. For any class $\mathcal{K}$,
 we denote by ${\overline {\mathcal{K}}} = \{{\overline f} :f\in \mathcal{K} \}$. The \emph{dual} of $f$,
 denoted $f^d$, is given by $f^d({\bf a})={\bar f}({\bar {\bf a}})$, for all ${\bf a } \in \mathbb{B}^n$. 
The dual of a class $\mathcal{K}$ of Boolean functions is defined as the set 
$\mathcal{K}^d=\{f^d:f\in \mathcal{K}\}$. We use ${\underline {\mathcal{K}}}$ to denote the class given by 
${\underline {\mathcal{K}}}={\overline {\mathcal{K}^d}}$.
  
\begin{fact}If $\mathcal{K}$ is an equational class, then $\overline {\mathcal{K}}$, $\mathcal{K}^d$ and 
${\underline {\mathcal{K}}}$ 
are also equational classes. In fact, $\mathcal{K}\mapsto {\overline {\mathcal{K}}}$,
$\mathcal{K}\mapsto \mathcal{K}^d$ and $\mathcal{K}\mapsto {\underline {\mathcal{K}}}$
are lattice automorphisms on the set ${\bf E}_\mathbb{B}$
of all equational classes of Boolean functions.  
 \end{fact}

It is well known that every Boolean function $f$ can be represented in the language of Boolean lattices
by a DNF expression (\emph{disjunctive normal form}), i.e. by an expression of the form
  \begin{displaymath}
\bigvee _{ i\in I}(\bigwedge _{j\in P_i}x_j \bigwedge _{j\in N_i}{\bar x}_j),
\end{displaymath}
 where $I$ is a finite, possibly empty, set of indices and each variable appears at 
most once in each conjunct.
 We regard empty disjunctions and empty conjunctions
as representing constant functions $\bf 0$ and $\bf 1$, respectively.
It is easy to verify that if 
  \begin{displaymath}
f=\bigvee _{ i\in I}(\bigwedge _{j\in P_i}x_j \bigwedge _{j\in N_i}{\bar x}_j),
\end{displaymath}
then the dual $f^d$ of $f$ is represented by
  \begin{displaymath}
f^d=\bigwedge _{ i\in I}(\bigvee _{j\in P_i}x_j \bigvee _{j\in N_i}{\bar x}_j) \qquad \qquad (1)
\end{displaymath} 
Expressions of the form $(1)$ are called CNF (\emph{conjunctive normal form}) representations.

 Since Stone \cite{S}, it is well-known that any Boolean lattice can be viewed as a Boolean ring 
(i.e. a commutative ring in which every element is idempotent under product)
 by defining multiplication and addition by
 \begin{displaymath}
x\cdot y =x\wedge y \textrm{ and } x\oplus y= (\bar {x}\wedge y)\vee (x\wedge \bar {y}).
\end{displaymath}
Thus both $\mathbb{B}^n$ and $\mathbb{B}^{\mathbb{B}^n}$ can also be treated as Boolean rings
by making use of the above algebraic translations. 
It is not difficult to see that each $n$-ary Boolean function $f$ can be represented
in this Boolean ring language by a multilinear polynomial in $n$ indeterminates over $\mathbb{B}$,  
called its \emph{Zhegalkin polynomial} or \emph{Reed-Muller polynomial}
 \begin{displaymath}
f={\Sigma  }_{ j\in I} (c_j\cdot  {\Pi }_{i\in I_j}x_i) 
\end{displaymath} 
Unlike DNF and CNF representations, the Zhegalkin polynomial representation
of a Boolean function is unique (up to permutation of terms and permutation of variables in the terms).
For further normal form representations of Boolean functions, see \cite{CFL}.

Recall that (\emph{Boolean}) \emph{clone} is a class 
$\mathcal{C}\subseteq \underset{n \geq 1}{\bigcup}  \mathbb{B}^{\mathbb{B}^n}$
idempotent under class composition and containing all projections. In the Boolean case,
 the only idempotent classes which are not clones
are exactly the empty class $\emptyset $, the class $C_0$ of constant $0$ functions,
the class $ C_1$ of constant $1$ functions, and the class $C$ containing all constants.  

The clones of Boolean functions
form an algebraic lattice by defining the 
 meet as the intersection of clones and the join as the smallest clone containing the union. 
 This lattice is known as Post Lattice (see Figure~\ref{PostLattice}),
 named after Emil Post who first described and classified in \cite{Post}
 the set of all Boolean clones (for recent and shorter proofs of Post's classification 
see \cite{RD}, \cite{U}, \cite{Z}; for general background see \cite{JGK} and \cite{Pi1}).
 We make use of notations and terminology appearing in \cite{FP} and in \cite{JGK}.  
 
 \begin{figure}[p]
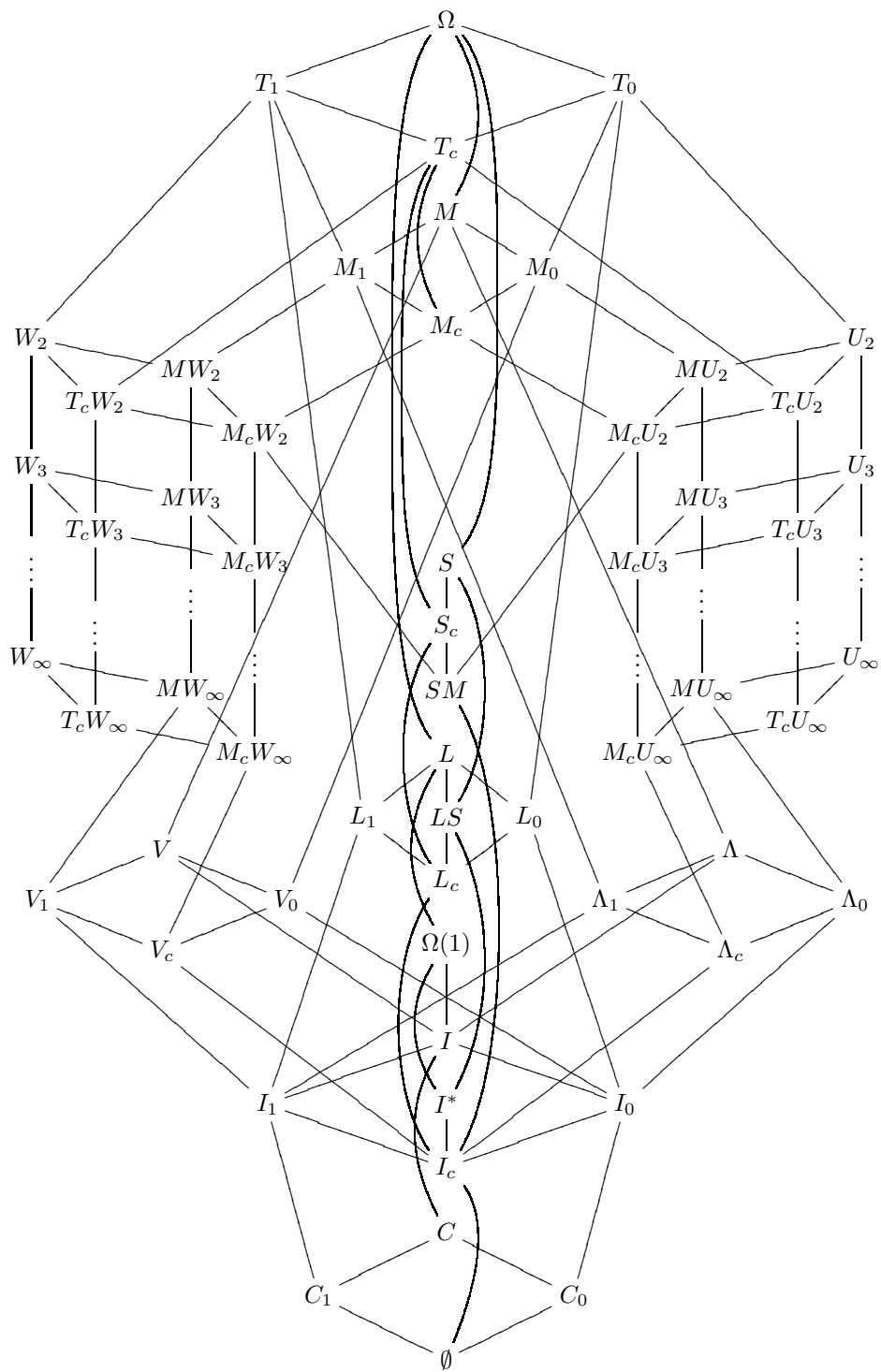

\[
\xy /r0.215pc/:,
(0,-30)*+{\emptyset }="E",
(20,-20)*+{C_0}="0",
(-20,-20)*+{C_1}="1",
(0,-10)*+{C}="O",
(0,0)*+{I_c}="Ic",
(0,10)*+{I^*}="Ia",
(28,10)*+{I_0}="I0",
(-28,10)*+{I_1}="I1",
(0,20)*+{I}="I",
(0,35)*+{\Omega(1)}="Omega1",
(44.5,34)*+{\Lambda_c}="Lambdac",
(64,42)*+{\Lambda_0}="Lambda0",
(25,42)*+{\Lambda_1}="Lambda1",
(44.5,50)*+{\Lambda}="Lambda",
(-44.5,34)*+{V_c}="Vc",
(-64,42)*+{V_1}="V1",
(-25,42)*+{V_0}="V0",
(-44.5,50)*+{V}="V",
(0,45)*+{L_c}="Lc",
(0,55)*+{LS}="LS",
(13,55)*+{L_0}="L0",
(-13,55)*+{L_1}="L1",
(0,65)*+{L}="L",
(0,75)*+{SM}="SM",
(0,85)*+{S_c}="Sc",
(0,95)*+{S}="S",
(0,132)*+{M_c}="Mc",
(15,141)*+{M_0}="M0",
(-15,141)*+{M_1}="M1",
(0,150)*+{M}="M",
(0,160)*+{T_c}="Tc",
(28,170)*+{T_0}="T0",
(-28,170)*+{T_1}="T1",
(0,180)*+{\Omega}="Omega",
(30,65)*+{M_cU_\infty}="McUinf",
(40,75)*+{MU_\infty}="MUinf",
(55,70)*+{T_cU_\infty}="TcUinf",
(65,80)*+{U_\infty}="Uinf",
(30,80)*+{\vdots}="McUdots",
(40,90)*+{\vdots}="MUdots",
(55,85)*+{\vdots}="TcUdots",
(65,95)*+{\vdots}="Udots",
(30,95)*+{M_cU_3}="McU3",
(40,105)*+{MU_3}="MU3",
(55,100)*+{T_cU_3}="TcU3",
(65,110)*+{U_3}="U3",
(30,115)*+{M_cU_2}="McU2",
(40,125)*+{MU_2}="MU2",
(55,120)*+{T_cU_2}="TcU2",
(65,130)*+{U_2}="U2",
(-30,65)*+{M_cW_\infty}="McWinf",
(-40,75)*+{MW_\infty}="MWinf",
(-55,70)*+{T_cW_\infty}="TcWinf",
(-65,80)*+{W_\infty}="Winf",
(-30,80)*+{\vdots}="McWdots",
(-40,90)*+{\vdots}="MWdots",
(-55,85)*+{\vdots}="TcWdots",
(-65,95)*+{\vdots}="Wdots",
(-30,95)*+{M_cW_3}="McW3",
(-40,105)*+{MW_3}="MW3",
(-55,100)*+{T_cW_3}="TcW3",
(-65,110)*+{W_3}="W3",
(-30,115)*+{M_cW_2}="McW2",
(-40,125)*+{MW_2}="MW2",
(-55,120)*+{T_cW_2}="TcW2",
(-65,130)*+{W_2}="W2",
"E";"0"**@{-},
"0";"I0"**@{-},
"E";"1"**@{-},
"1";"I1"**@{-},
"1";"O"**@{-},
"0";"O"**@{-},
"Ic";"Ia"**@{-},
"Ic";"I0"**@{-},
"Ic";"Ia"**@{-},
"Ic";"I1"**@{-},
"Ic";"Lambdac"**@{-},
"Ic";"Vc"**@{-},
"I0";"I"**@{-},
"I0";"Lambda0"**@{-},
"I0";"V0"**@{-},
"I0";"L0"**@{-},
"I1";"I"**@{-},
"I1";"V1"**@{-},
"I1";"Lambda1"**@{-},
"I1";"L1"**@{-},
"I";"Omega1"**@{-},
"I";"Lambda"**@{-},
"I";"V"**@{-},
"Lambdac";"Lambda0"**@{-},
"Lambdac";"Lambda1"**@{-},
"Lambda0";"Lambda"**@{-},
"Lambda1";"Lambda"**@{-},
"Lambdac";"McUinf"**@{-},
"Lambda0";"MUinf"**@{-},
"Lambda1";"M1"**@{-},
"Lambda";"M"**@{-},
"Vc";"V0"**@{-},
"Vc";"V1"**@{-},
"V0";"V"**@{-},
"V1";"V"**@{-},
"Vc";"McWinf"**@{-},
"V1";"MWinf"**@{-},
"V0";"M0"**@{-},
"V";"M"**@{-},
"Lc";"L0"**@{-},
"Lc";"L1"**@{-},
"Lc";"LS"**@{-},
"L0";"L"**@{-},
"L1";"L"**@{-},
"LS";"L"**@{-},
"L0";"T0"**@{-},
"L1";"T1"**@{-},
"SM";"Sc"**@{-},
"Sc";"S"**@{-},
"SM";"McU2"**@{-},
"SM";"McW2"**@{-},
"Mc";"M0"**@{-},
"Mc";"M1"**@{-},
"M0";"M"**@{-},
"M1";"M"**@{-},
"M0";"T0"**@{-},
"M1";"T1"**@{-},
"Tc";"T0"**@{-},
"Tc";"T1"**@{-},
"T0";"Omega"**@{-},
"T1";"Omega"**@{-},
"McUinf";"TcUinf"**@{-},
"McUinf";"MUinf"**@{-},
"TcUinf";"Uinf"**@{-},
"MUinf";"Uinf"**@{-},
"McUinf";"McUdots"**@{-},
"MUinf";"MUdots"**@{-},
"TcUinf";"TcUdots"**@{-},
"Uinf";"Udots"**@{-},
"McUdots";"McU3"**@{-},
"MUdots";"MU3"**@{-},
"TcUdots";"TcU3"**@{-},
"Udots";"U3"**@{-},
"McU3";"TcU3"**@{-},
"McU3";"MU3"**@{-},
"TcU3";"U3"**@{-},
"MU3";"U3"**@{-},
"McU3";"McU2"**@{-},
"MU3";"MU2"**@{-},
"TcU3";"TcU2"**@{-},
"U3";"U2"**@{-},
"McU2";"TcU2"**@{-},
"McU2";"MU2"**@{-},
"TcU2";"U2"**@{-},
"MU2";"U2"**@{-},
"McU2";"Mc"**@{-},
"MU2";"M0"**@{-},
"TcU2";"Tc"**@{-},
"U2";"T0"**@{-},
"McWinf";"TcWinf"**@{-},
"McWinf";"MWinf"**@{-},
"TcWinf";"Winf"**@{-},
"MWinf";"Winf"**@{-},
"McWinf";"McWdots"**@{-},
"MWinf";"MWdots"**@{-},
"TcWinf";"TcWdots"**@{-},
"Winf";"Wdots"**@{-},
"McWdots";"McW3"**@{-},
"MWdots";"MW3"**@{-},
"TcWdots";"TcW3"**@{-},
"Wdots";"W3"**@{-},
"McW3";"TcW3"**@{-},
"McW3";"MW3"**@{-},
"TcW3";"W3"**@{-},
"MW3";"W3"**@{-},
"McW3";"McW2"**@{-},
"MW3";"MW2"**@{-},
"TcW3";"TcW2"**@{-},
"W3";"W2"**@{-},
"McW2";"TcW2"**@{-},
"McW2";"MW2"**@{-},
"TcW2";"W2"**@{-},
"MW2";"W2"**@{-},
"McW2";"Mc"**@{-},
"MW2";"M1"**@{-},
"TcW2";"Tc"**@{-},
"W2";"T1"**@{-},
"E";"Ic"**\crv{(10,-7)},
"Ic";"Lc"**\crv{(-10,7)&(-10,38)},
"Ic";"SM"**\crv{(11,10)&(11,65)},
"Ia";"Omega1"**\crv{(-10,22.5)},
"Ia";"LS"**\crv{(8,22)&(8,43)},
"O";"I"**\crv{(-10,7)},
"Omega1";"L"**\crv{(-11,50)},
"Lc";"Sc"**\crv{(-9,52)&(-9,78)},
"LS";"S"**\crv{(8,63)&(8,87)},
"L";"Omega"**\crv{(-8.5,70)&(-8.5,122.5)&(-8.5,175)},
"Sc";"Tc"**\crv{(-7,90)&(-7,122.5)&(-7,155)},
"S";"Omega"**\crv{(8,100)&(8,137.5)&(8,175)},
"Mc";"Tc"**\crv{(-9,146)},
"M";"Omega"**\crv{(10,165)}
\endxy
\]
\caption{\label{PostLattice}Post Lattice.}
\end{figure}

 \begin{itemize}
\item $\Omega $ denotes the class $\underset{n \geq 1}{\bigcup} \mathbb{B}^{\mathbb{B}^n}$ of all Boolean functions; 
\item $T_0$ and $T_1$ denote the classes of $0$- and $1$-preserving functions, respectively, 
i.e.,\newline 
$T_0 = \{f \in \Omega : f(0, \ldots, 0) = 0\}$, $T_1 = \{f \in \Omega : f(1, \ldots, 1) = 1\}$;

\item $T_c$ denotes the class of constant-preserving functions, i.e.,
$T_c = T_0 \cap T_1$.

\item $M$ denotes the class of all monotone functions, i.e.,\newline
$M = \{f \in \Omega : \text{$f( {\bf a}) \leq f( {\bf b})$, whenever 
${\bf a} \preceq {\bf b}$}\}$;

\item $M_0 = M \cap T_0$, $M_1 = M \cap T_1$, $M_c = M \cap T_c$; 

\item $S$ denotes the class of all self-dual functions, i.e.,\newline
$S = \{f \in \Omega : f^d = f\}$;

\item $S_c = S \cap T_c$, $SM = S \cap M$;

\item $L$ denotes the class of all linear functions, i.e.,\newline
$L = \{f \in \Omega : \text{$f = c_0 {\bf 1} + c_1 x_1 + \dots + c_n x_n$ for some $n$
 and $c_0, \ldots, c_n \in \mathbb{B}$}\}$;

\item $L_0 = L \cap T_0$, $L_1 = L \cap T_1$, $LS = L \cap S$, $L_c = L \cap T_c$;
\end{itemize}

\noindent
Let $a \in \{0,1\}$. A set $A \subseteq \{0,1\}^n$ is said to be \emph{$a$-separating} 
if there is $i$, $1\leq i\leq n$, such that for every $(a_1, \ldots, a_n) \in A$ we have $a_i = a$.
 A function $f$ is said to be \emph{$a$-separating} if $f^{-1}(a)$ is $a$-separating. 
The function $f$ is said to be \emph{$a$-separating of rank $k \geq 2$} 
if every subset $A \subseteq f^{-1}(a)$ of size at most $k$ is $a$-separating.

\begin{itemize}
\item For $m \geq 2$, $U_m$ and $W_m$ denote the classes of all $1$- and $0$-separating 
functions of rank $m$, respectively;   

\item $U_\infty$ and $W_\infty$ denote the classes of all $1$- and $0$-separating functions,
 respectively, i.e., $U_\infty = \bigcap_{k \geq 2} U_k$ and $W_\infty = \bigcap_{k \geq 2} W_k$;

\item $T_cU_m = T_c \cap U_m$ and $T_cW_m = T_c \cap W_m$, for $m = 2, \ldots, \infty$;

\item $MU_m = M \cap U_m$ and $MW_m = M \cap W_m$, for $m = 2, \ldots, \infty$;

\item $M_cU_m = M_c \cap U_m$ and $M_cW_m = M_c \cap W_m$, for $m = 2, \ldots, \infty$;

\item $\Lambda $ denotes the class of all conjunctions and constants, i.e.,\newline
$\Lambda = \{f \in \Omega : f = {\bf 0}, {\bf 1}, x_{i_1} \wedge \dots \wedge x_{i_n}
\text{ for some $n \geq 1$ and $i_j$'s}\}$;

\item $\Lambda_0 = \Lambda \cap T_0$, $\Lambda_1 = \Lambda \cap T_1$, $\Lambda_c = \Lambda \cap T_c$; 

\item $V$ denotes the class of all disjunctions and constants, i.e.,\newline
$V = \{f \in \Omega : \text{$f = {\bf 0}, {\bf 1}, x_{i_1} \vee \dots \vee x_{i_n}$ 
for some $n \geq 1$ and $i_j$'s}\}$;

\item $V_0 = V \cap T_0$, $V_1 = V \cap T_1$, $V_c=V \cap T_c$;

\item $\Omega (1)$ denotes the class of all quasi-monadic functions, i.e. variables, negated variables, and constants;

\item $I^*$ denotes the class of all variables and negated variables;

\item $I$ denotes the class of all variables and constants;

\item $I_0 = I \cap T_0$, $I_1 = I \cap T_1$; 

\item $I_c$ denotes the smallest clone containing only variables, i.e., $I_c = I \cap T_c$.
\end{itemize}

Since there are essentially 4 quasi-monadic Boolean functions,
 namely $\{x_1,{\bar x_1}, {\bf 0}, {\bf 1}\}$, and since
$\Omega (1)= \{x_1,{\bar x_1}, {\bf 0}, {\bf 1}\}I_c= x_1I_c\cup {\bar x_1}I_c \cup {\bf 0}I_c \cup {\bf 1}I_c$,
 we have:
  
\begin{theorem} 
There are exactly $16$ equational classes in $\interval{\emptyset }{\Omega (1)}$. 
\end{theorem}

\begin{figure}[thp]
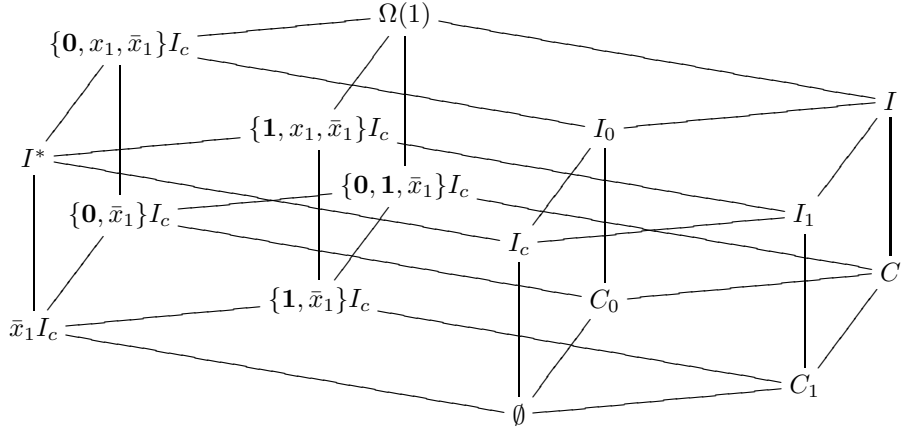

\[
\xy /r0.180pc/:,
(10,10)*+{\emptyset }="b",
(25,30)*+{ C_0}="b0",
(60,15)*+{C_1}="b1",
(75,35)*+{ C}="b01",
(10,40) *+{ I_c}="bx",
(25,60) *+{ I_0}="b0x",
(60,45) *+{ I_1}="b1x",
(75,65) *+{ I}="b01x",
(-75,25)*+{{\bar x_1} I_c}="bn",
(-60,45)*+{ \{{\bf 0},{\bar x_1}\} I_c}="bn0",
(-25,30)*+{ \{{\bf 1},{\bar x_1}\} I_c}="bn1",
(-10,50)*+{ \{{\bf 0},{\bf 1},{\bar x_1}\} I_c}="bn01",
(-75,55) *+{I^*}="bnx",
(-60,75) *+{ \{{\bf 0},x_1,{\bar x_1}\}  I_c}="bn0x",
(-25,60) *+{\{{\bf 1},x_1,{\bar x_1}\} I_c}="bn1x",
(-10,80) *+{\Omega (1)}="bn01x",
"b";"b0"**@{-},
"b";"b1"**@{-},
"b";"bx"**@{-},
"b";"bn"**@{-},
"b0";"b01"**@{-},
"b0";"b0x"**@{-},
"b0";"bn0"**@{-},
"b1";"b01"**@{-},
"b1";"b1x"**@{-},
"b1";"bn1"**@{-},
"bx";"b0x"**@{-},
"bx";"b1x"**@{-},
"bx";"bnx"**@{-},
"bn";"bn0"**@{-},
"bn";"bn1"**@{-},
"bn";"bnx"**@{-},
"b01";"b01x"**@{-},
"b01";"bn01"**@{-},
"b0x";"b01x"**@{-},
"b0x";"bn0x"**@{-},
"b1x";"b01x"**@{-},
"b1x";"bn1x"**@{-},
"bn0";"bn0x"**@{-},
"bn0";"bn01"**@{-},
"bn1";"bn01"**@{-},
"bn1";"bn1x"**@{-},
"bnx";"bn0x"**@{-},
"bnx";"bn1x"**@{-},
"bn01";"bn01x"**@{-},
"b01x";"bn01x"**@{-},
"bn0x";"bn01x"**@{-},
"bn1x";"bn01x"**@{-},
\endxy
\]
\caption{\label{FinLattice0}Lattice of equational classes containing only quasi-monadic functions.}
\end{figure}

Looking at Figure~\ref{PostLattice}, we see that the Post Lattice is co-atomic, that is,
 every clone is contained in a \emph{maximal} clone (co-atom).
In fact, for any finite set $A$, the lattice of clones on $A$ is co-atomic, 
and the number of maximal clones (co-atoms) is known to be finite (see \cite{Ro}). 
This is not the case in the lattice of equational classes. 

\begin{theorem}
The lattice ${\bf E}_{\mathbb{B}} $ has no co-atoms.
\end{theorem}

\begin{proof}
For a contradiction, suppose that ${\bf E}_{\mathbb{B}} $ has a co-atom, say $\mathcal{M}$.
Let $f\in \Omega \setminus \mathcal{M}$. If
\begin{itemize}
\item $f={x}$, then $\mathcal{M} \cap L_c=\emptyset $,
\item $f={\bar x}$, then $\mathcal{M} \cap (LS\setminus L_c)=\emptyset $,
\item $f={\bf 0}$, then $\mathcal{M} \cap (L_0\setminus L_c)=\emptyset $,
\item $f={\bf 1}$, then $\mathcal{M} \cap (L_1\setminus L_c)=\emptyset $,
\end{itemize} 
and thus 
$\mathcal{M} \subset \mathcal{M} \cup f I_c\subset \mathcal{M} \cup \{f, f'\} I_c\subseteq \Omega $,
 for a suitable $f'$ in e.g. $\{x_1+x_2+x_3 ,x_1+x_2+x_3+1, x_1+x_2, x_1+x_2+1\}$,
 contradicting our assumption.

So let $f\not=x,{\bar x},{\bf 0},{\bf 1}$ be of essential arity $n\geq 2$. Without loss of generality,
assume that $f$ has no dummy variables.
Now consider $f'=x+y+f$, where $x$ and $y$ are not essential variables of $f$. Obviously,
$f'\not\preceq _{\bf V}f$. Furthermore, $f'\not\in \mathcal{M} $, otherwise,
 by identifying $x=y$ we would have $f\in \mathcal{M} $.
Hence, $\mathcal{M} \subset \mathcal{M} \cup f I_c\subset \Omega $, which yields the desired contradiction.
\end{proof}

\subsection{Antichains of Boolean functions} 
In the sequel, we will make use of the following fact.
\begin{fact}
\label{fact:dual} Let $\mathcal{C}_1$ and $\mathcal{C}_2 $ be idempotent classes such that 
$\mathcal{C}_1\subseteq  \mathcal{C}_2$. If $(f_i)_{i\in I}$ is an antichain in 
$\mathcal{C}_2\setminus \mathcal{C}_1$, then $( {\overline {f_i}} )_{i\in I}$,
$(f^d_i)_{i\in I}$ and $({\overline {f^d_i}} )_{i\in I}$
 are antichains in ${\overline {\mathcal{C}_2}} \setminus {\overline {\mathcal{C}_1}} $,
 $\mathcal{C}^d_2\setminus \mathcal{C}^d_1$ and 
${\underline {\mathcal{C}_2}} \setminus {\underline {\mathcal{C}_1}}$,
 respectively.
\end{fact}

Lemma~\ref{lemma1} is a particular case of Proposition 3.4 in \cite{Pi2}.
\begin{lemma}\label{lemma1}
The family $(f_n)_{n\geq 4}$ of $0$-preserving Boolean functions, given by 
\begin{displaymath}
f_n(x_1,\ldots ,x_n)=\left\{
             \begin{array}{ll}
                     1         &       \mbox{if $\#\{i:x_i=1\}\in \{1,n-1\}$}   \\
                     0         &       \mbox{otherwise}.                
             \end{array}\right.
 \end{displaymath}
constitutes an (infinite) antichain of Boolean functions, i.e. if $m\not=n$, then 
$f_m\not\preceq _{\bf V}f_n$ and $f_n\not\preceq _{\bf V}f_m$.
\end{lemma}

\begin{lemma}\label{lemma2}
The family $(g_n)_{n\geq 4}$ of constant-preserving Boolean functions, given by 
\begin{displaymath}
g_n(x_1,\ldots ,x_n)=\left\{
             \begin{array}{ll}
                     0         &       \mbox{if $\#\{i:x_i=0\}\in \{1,n\}$}   \\
                     1         &       \mbox{otherwise}.                
             \end{array}\right.
 \end{displaymath} 
constitutes an (infinite) antichain of Boolean functions.
\end{lemma}

\begin{proof} To prove the lemma, it is enough to show that if $m\not=n$, then 
$g_m\not\preceq _{\bf V}g_n$. 
 
 By definition, $g_m$ and $g_n$ cannot have dummy variables, and hence,
 $g_m\not\preceq _{\bf V}g_n$, whenever $m>n$.
 So suppose that $m<n$.  
Note that for every $t\geq 4$, $g_t$ is constant with value 1
 on all $t$-tuples with at least two zeros and at least one 1.
For a contradiction,
 suppose that $g_m\preceq _{\bf V}g_n$, i.e.
there are $m$-ary projections $p_1,\ldots ,p_n\in I_c$ such that 
$g_m=g_n(p_1,\ldots ,p_n)$. Since every variable of $g_m$ is essential in $g_n$ and $m<n$,
 it follows that there are 
 at least two indices $1\leq i<j\leq n$ such that $p_i=p_j$. Also, since $4\leq m$, there is at least
 one index $1\leq k\leq n$ such that $p_k\not=p_i=p_j$. Now, consider the set $P$ of all 
$m$-tuples $(a_1,\ldots ,a_m)$ such that $p_i(a_1,\ldots ,a_m)=p_j(a_1,\ldots ,a_m)=0$, and 
$p_k(a_1,\ldots ,a_m)=1$.
 Clearly, $g_m$ is not constant because $P$ contains an $m$-tuple with exactly one 0,
 and an $m$-tuple with two 0's. 
But $g_n$ is constant 
with value 1 on all $n$-tuples of the form $(p_1(a_1,\ldots ,a_m),\ldots ,p_n(a_1,\ldots ,a_m))$, for 
$(a_1,\ldots ,a_m)\in P$,
 because all 
$n$-tuples of this form have at least two 0's and at least one 1, which yields  the desired contradiction. 
\end{proof}

\begin{lemma}\label{lemma3}
Let $(f_n)_{n\geq 4}$ and $(g_n)_{n\geq 4}$ be the families of Boolean functions
given above, and consider the families $(u_n)_{n\geq 4}$ and $(t^u_n)_{n\geq 4}$ defined by  
 \begin{displaymath} 
\begin{array}{llll}
    u_n(x_0,x_1,\ldots ,x_n) =x_0\wedge f_n(x_1,\ldots ,x_n)     \\
    t^u_n(x_0,x_1,\ldots ,x_n) =x_0\wedge g_n(x_1,\ldots ,x_n)     \\
             \end{array}
 \end{displaymath} 
 Each of $(u_n)_{n\geq 4}$ and $(t^u_n)_{n\geq 4}$ constitutes an (infinite) antichain of Boolean functions.
\end{lemma}

\begin{proof} We only prove that the lemma holds for the family $(u_n)_{n\geq 4}$.
The remaining claim can be shown to hold, by proceeding similarly.

We show that if $m\not=n$, then
$u_m\not\leq u_n$.
 By definition, $u_m$ and $u_n$ cannot have dummy variables. Therefore,
 $u_m\not\leq u_n$, whenever $m>n$.

 So assume that $m<n$, and for a contradiction, suppose that $u_m\preceq _{\bf V} u_n$, i.e.
there are $m+1$-ary projections $p_0,\ldots ,p_n\in I_c$ such that
$u_m=u_n(p_0,\ldots ,p_n)$. Note that for every $m\geq 4$,
$u_m(1,x_1\ldots ,x_m)=f_m(x_1\ldots ,x_m)$ and $u_m(0,x_1\ldots ,x_m)$ is the constant 0.

Now, suppose that $p_0(x_0,\ldots ,x_m)=x_0$. If for all $k\in {\bf n}$, $p_k(x_0,\ldots ,x_m)\not=x_0$,
then by taking $x_0=1$ we would conclude that $f_m\preceq _{\bf V} f_n$,
contradicting Lemma \ref{lemma1}. 

Suppose that there is $k\in {\bf n}$ such that $p_k(x_0,\ldots ,x_m)=x_0$. From the fact that each variable of $u_m$ is essential, it follows that for each $ j\in {\bf m}$ there is $l\in {\bf n}$ such that
 $p_l(x_0,\ldots ,x_m)=x_j$. Hence, by taking $a_i=1$ if and only if $i=0,1$, we have that the vector $(p_1(a_0,\ldots  ,a_m),\ldots ,p_n(a_0,\ldots  ,a_m))$ has at least $2$ and at most $n-2$ components equal to $1$ and thus
 \begin{displaymath}
 \begin{array}{ll}
u_m(a_0,\ldots  ,a_m)=1\not=0=
u_n(a_0,p_1(a_0,\ldots  ,a_m),\ldots ,p_n(a_0,\ldots  ,a_m))
 \end{array}
 \end{displaymath}
which is also a contradiction.

Hence, $p_0(x_0,\ldots ,x_m)\not=x_0$, say $p_0(x_0,\ldots ,x_m)=x_j$ for $j\in {\bf m}$. But then by taking $a_i=1$ if and only if $i=0,k$, for some $k\in {\bf m}$ such that $k\not=j$, we would have
\begin{displaymath}
 \begin{array}{ll}
u_m(a_0,a_1,\ldots a_{k-1},a_k,a_{k+1},\ldots  ,a_m)=u_m(1,0,\ldots,0 ,1, 0,\ldots ,0)=\\
f_m(0,\ldots,0 ,1,0,\ldots ,0)=1\not=
0=\\
u_n(0,p_1(1,0,\ldots,0 ,1, 0,\ldots ,0),\ldots ,p_n(1,0,\ldots,0 ,1, 0,\ldots ,0))=\\
u_n(a_j,p_1(a_0,\ldots  ,a_m),\ldots ,p_n(a_0,\ldots ,a_m))
 \end{array}
 \end{displaymath}
 which contradicts our assumption $u_m\preceq _{\bf V}u_n$.    
\end{proof}

A \emph{hypergraph} is an ordered pair $G=(V,E)$, where $V=V(G)$ is a non-empty finite set 
(called the \emph{set of vertices} of $G$), and $E=E(G)$ is a set of subsets of $V$
 (called the \emph{set of hyperedges} of $G$). Without loss of generality, we assume that our 
 hypergraphs $G$ have set of vertices $V(G)={\bf n}=\{1,\ldots ,n\}$, for some positive integer $n$.  
Examples of hypergraphs are the \emph{complete graphs} $K_n$, $n\geq 2$, whose set of vertices is 
$V(K_n)={\bf n}$ and whose set of hyperedges is the set of all 2-element subsets of $V(K_n)$, i.e. 
$E(K_n)=\{\{i,j\}:i,j\in V(K_n), i\not=j\}$.
To each hypergraph $G$, say $V(G)={\bf n}$, we associate an $n$-ary monotone
 Boolean function $f_G$ whose DNF is given by 
 $$ f_G=\bigvee _{I\in E(G)}\bigwedge _{i\in I}x_i$$
Note that every monotone Boolean function is associated with some hypergraph. 

Given two hypergraphs $G$ and $H$, a \emph{homomorphism} $h$ from $G$ to $H$ is any mapping
$h: V (G)\rightarrow   V (H)$ satisfying the condition:
if $I \in  E(G)$, then $h(I)=\{h(i):i\in I\} \in  E(H)$.
A homomorphism $h: V (G)\rightarrow   V (H)$ is said to be \emph{hyperedge-surjective} if for each 
$J \in  E(H)$, there is $I \in  E(G)$ such that $I= h^{-1}(J)$.

The following lemma provides a characterization of $\preceq _{\bf V} $ restricted
 to the clone $M$ of monotone Boolean functions.
\begin{lemma}\label{lemma4} 
Let $G$ and $H$ be two hypergraphs, and consider the functions $f_G$ and $f_H$ associated with $G$ and $H$, respectively.
Then there is a hyperedge-surjective homomorphism $f: V (G)\rightarrow  V (H)$ if and only if 
$f_H\preceq _{\bf V} f_G$.
\end{lemma}
\begin{proof} Let $V (G)=\bf n$ and $V (H)=\bf m$. Assume first that there is a hyperedge-surjective
 homomorphism $h: V (G)\rightarrow   V (H)$.
Define $m$-ary projections $p_1,\ldots ,p_n\in I_c$ by
$p_i=x_j$ if and only if $h(i)=j$. Consider the $m$-ary function $g$ given by 
$g=f_G(p_1,\ldots ,p_n)$. Note that
 \begin{displaymath}
g= \bigvee _{ I\in E(G) } \bigwedge _{ i\in I } p_i =\bigvee _{ I\in E(G) } \bigwedge _{ j\in h(I) } x_j 
 \end{displaymath} 
Now, since $h$ is a hyperedge-surjective homomorphism, we have that for each $I\in E(G)$, $h(I)\in E(H)$, and that 
every $J\in E(H)$ is of the form $h(I)$, for some $I\in E(G)$. Also, both $\vee $ and $\wedge $ are 
associative and idempotent operations, and thus
 \begin{displaymath}
g=\bigvee _{I\in E(G)} \bigwedge _{j\in h(I)} x_j = \bigvee _{J\in E(H)} \bigwedge _{j\in J} x_j =f_H
 \end{displaymath} 
In other words, $f_H\preceq _{\bf V} f_G$. 

Now, suppose that $f_H\preceq _{\bf V} f_G$, i.e. there are $m$-ary projections $p_1,\ldots ,p_n\in I_c$ 
such that $f_H=f_G(p_1,\ldots ,p_n)$.  Let $h$ be the map $h: V (G)\rightarrow   V (H)$ satisfying 
$h(i)=j$ if and only if $p_i=x_j$. We claim that $h$ is a homomorphism.
Indeed, if $I \in  E(G)$, then  $\underset{i\in I}{\bigwedge} x_i$ is conjunct of $f_G$, and thus 
$\underset{i\in I}{\bigwedge}x_{h(i)}=\underset{j\in h(I)}{\bigwedge}x_{j}$ is a conjunct of $f_H$.
By definition of $f_H$, we have that $h(I)\in E(H)$.
To see that $h$ is hyperedge-surjective, suppose that $J \in  E(H)$. Then $\underset{j\in J}{\bigwedge}x_{j}$   
is a conjunct of $f_H$. By construction, we have that there is $I\subseteq V(G)$ such that $I=h^{-1}(J)$ and
$\underset{i\in I}{\bigwedge} x_i$ is a conjunct of $f_G$. By definition of $f_G$, it follows that $I \in  E(G)$, 
and the proof of the lemma is complete. 
\end{proof}

\begin{lemma}\label{lemma5}
The family $(H_n)_{n\geq 2}$ of constant-preserving monotone Boolean functions given by 
\begin{displaymath}
H_n(x_1,\ldots ,x_n)= \bigvee _{1\leq i<j\leq n}x_i\wedge x_j. 
 \end{displaymath} 
constitutes an (infinite) antichain of Boolean functions.
 Furthermore, for each $n\geq 2$, the family $(G^n_m)_{m\geq n}$ of 
composites
\begin{displaymath}
G^n_m(x_1,\ldots ,x_{m+n-1})=H_n(x_1,\ldots ,x_{n-1},H_m(x_n,\ldots ,x_{m+n-1})) 
 \end{displaymath}
 also constitutes an (infinite) antichain of Boolean functions.
\end{lemma}

\begin{proof} To see that the first claim of the lemma holds, observe that for each $n\geq 2$, $H_n$ is
the $n$-ary function associated with the complete graph $K_n$. Since there is no 
hyperedge-surjective homomorphism between $K_m$ and $K_n$,
 whenever $m\not=n$, by Lemma~\ref{lemma4} it follows that  
$H_m$ and $H_n$ are incomparable, whenever $m\not=n$.

To prove that the second claim of the lemma also holds, we show that if $m_1\not=m_2$, then 
$G^n_{m_1} \not\preceq _{\bf V}G^n_{m_2}$. Note first that each 
$G^n_{m}$ is associated with a hypergraph $G$ whose set of vertices is $\{1,\ldots ,m+n-1\}$ and whose set of 
hyperedges is 
\begin{displaymath}
E(G)=\{\{i,j\}: 1\leq i<j\leq n-1\}\cup \{\{i,k,l\}: 1\leq i\leq n-1, n\leq k<l\leq m+n-1\}.
 \end{displaymath} 
Now, if $m_1>m_2$, then $G^n_{m_1}$ and $G^n_{m_2}$ are associated with graphs $G_1$ and $G_2$,
 respectively, such that $G_1$ and $G_2$ have the same number of 2-element hyperedges,
but $G_1$ has more 3-element hyperedges than $G_2$. 
From this fact it follows that if $m_1>m_2$, then there is
 no hyperedge-surjective homomorphism $h:G_2\rightarrow G_1$,
 and $G^n_{m_1} \not\preceq _{\bf V}G^n_{m_2}$ by Lemma~\ref{lemma4}.  

Now, suppose that $m_1<m_2$ and for a contradiction suppose that there is a hyperedge-surjective homomorphism $h:G_2\rightarrow G_1$.
Clearly, each 2-element hyperedge of $G_2$ must be mapped to a 2-element hyperedge of $G_1$, and since 
$h$ is hyperedge-surjective and $G_1$ and $G_2$ have the same number of 2-element hyperedges,  there cannot be two 2-element hyperedges of $G_2$
mapped to the same 2-element hyperedge of $G_1$. Also, no 3-element hyperedge of $G_2$ can be mapped to a  
2-element hyperedge $J\in E(G_1)$, for otherwise $h^{-1}(J)$ would be of size at least $4$ and there is 
no hyperedge of $G_2$ of size greater than $3$. Similarly, there cannot be two
 3-element hyperedges of $G_2$ mapped to the same 3-element hyperedge of $G_1$. 
But then there is a 3-element hyperedge $I\in E(G_2)$ such that $h(I)\not\in E(G_1)$, 
which contradicts our assumption that $h:G_2\rightarrow G_1$ is a homomorphism.

Thus, if $m_1<m_2$, then there is
 no hyperedge-surjective homomorphism $h:G_2\rightarrow G_1$,
 and by Lemma~\ref{lemma4} it follows that $G^n_{m_1}\not\preceq _{\bf V}G^n_{m_2}$,
which completes the proof of the lemma.
\end{proof}

\begin{lemma} \label{lemma7}Let $\mathbb{O}$ denote the set of all odd integers $n\geq 7$, and
let $\mu _n$ denote the $n$-ary threshold function defined by 
\begin{displaymath}
\mu _n(x_1,\ldots ,x_n)=\left\{
             \begin{array}{ll}
                     1         &       \mbox{if $\#\{i:x_i=1\}\geq {\frac{n+1}{2}} \}$}   \\
                     0         &       \mbox{otherwise}.                
             \end{array}\right.
 \end{displaymath} 
The family $(T_n)_{n\in \mathbb{O}}$ given by 
\begin{displaymath}
T_n(x_1,\ldots ,x_n,x_{n+1}, x_{n+2} )=\left\{
             \begin{array}{lll}
             H_n (x_1,\ldots ,x_n)   &       \mbox{if $x_{n+1}=x_{n+2}=1 $}   \\
             \mu _n(x_1,\ldots ,x_n)        &       \mbox{if $x_{n+1}+x_{n+2}=1$ }\\                
             H^d_n (x_1,\ldots ,x_n)   &       \mbox{if $x_{n+1}=x_{n+2}=0 $}   
\end{array}\right.
 \end{displaymath} 
constitutes an (infinite) antichain of Boolean functions.
 Moreover, the family
$(s_n)_{n\in \mathbb{O}}$ defined by
\begin{displaymath}
s_n (x_1,\ldots ,x_n,x_{n+1}, x_{n+2} )=T_n(x_1,\ldots ,x_n,{\bar x_{n+1}},{\bar x_{n+2}})
 \end{displaymath} 
also constitutes an (infinite) antichain of Boolean functions.
\end{lemma}

\begin{proof} Since 
each function of $({T_n})_{n\in \mathbb{O}}$, and 
each function of $(s_n)_{n\in \mathbb{O}}$ has only essential variables,
to prove the lemma we only need to show that if $n<m$, then $T_n\not\preceq _{\bf V} T_m$ and
 $s_n\not\preceq _{\bf V}s_m$.

So assume that $7\leq n<m$, and for a contradiction, suppose first that $T_{n}\preceq _{\bf V}T_m$, i.e.   
there are $n+2$-ary projections $p_1,\ldots ,p_{m+2}$ such that 
 \begin{displaymath}
T_n=T_m(p_1,\ldots ,p_{m+2}). 
 \end{displaymath}
Note that for each $1\leq i\leq n+2$, there is at least one $1\leq i_1\leq m+2$ such that $p_{i_1}=x_i$
because $T_n$ has no dummy variables.

First we consider the case $p_{m+1}=p_{m+2}$. Let $1\leq j\leq n$.
 If both projections are $x_{n+1}$, or $x_{n+2}$, or $x_j$, then 
 for $a_i=1$ if and only if $i=j, n+1,n+2$, we have
 \begin{displaymath}
 \begin{array}{lll}  
T_{n} (a_1,\ldots ,a_{n+2}) =H_n (a_1,\ldots ,a_{n}) =0 \textrm{ and }\\
T_{m} (p_1(a_1,\ldots ,a_{n+2}),\ldots ,p_{m+2}(a_1,\ldots ,a_{n+2}))=\\
H_m (p_1(a_1,\ldots ,a_{n+2}),\ldots ,p_{m}(a_1,\ldots ,a_{n+2}))=1
 \end{array}
 \end{displaymath}
 
If $p_{m+1} \not= p_{m+2} $, say $p_{m+1}=x_j$ and $p_{m+2}=x_k $, $1\leq j<k\leq n$,
then for $a_i=1$ if and only if $i\not=i_1,i_2,j,k$, where 
$1\leq i_1<i_2\leq n$ are indices distinct from $j$ and $k$,
 we have
 \begin{displaymath}
 \begin{array}{lll}  
T_{n} (a_1,\ldots ,a_{n+2}) =H_n (a_1,\ldots ,a_{n}) =1\textrm{ and }\\
T_{m} (p_1(a_1,\ldots ,a_{n+2}),\ldots ,p_{m+2}(a_1,\ldots ,a_{n+2}))=\\
H^d_m (p_1(a_1,\ldots ,a_{n+2}),\ldots ,p_{m}(a_1,\ldots ,a_{n+2}))=0
 \end{array}
 \end{displaymath}

Next we consider the case $p_{m+1} \in \{x_{n+1},x_{n+2}\}$ and $ p_{m+2} \in \{x_{1},\ldots ,x_{n}\}$,
or $p_{m+2} \in \{x_{n+1},x_{n+2}\}$ and $ p_{m+1} \in \{x_{1},\ldots ,x_{n}\}$.
Without loss of generality, assume that 
$p_{m+1}=x_{n+1}$ and $ p_{m+2}=x_{j}$, $1\leq j\leq n$.
If there are at least two $1\leq i_1<i_2\leq m$ such that $p_{i_1}=p_{i_2}=x_{n+2}$, then 
for $a_i=1$ if and only if $i=j,n+1,n+2$, 
 \begin{displaymath}
 \begin{array}{lll}  
T_{n} (a_1,\ldots ,a_{n+2}) =H_n (a_1,\ldots ,a_{n}) =0\textrm{ and }\\
T_{m} (p_1(a_1,\ldots ,a_{n+2}),\ldots ,p_{m+2}(a_1,\ldots ,a_{n+2}))=\\
H_m (p_1(a_1,\ldots ,a_{n+2}),\ldots ,p_{m}(a_1,\ldots ,a_{n+2}))=1
 \end{array}
 \end{displaymath}
If there is a unique $1\leq k\leq m$ such that $p_{k}=x_{n+2}$, then 
let $I=\{i_1,\ldots ,i_{\frac{n+1}{2}} \}$ be a ``majority" of indices not containing $j,n+1,n+2$.
Thus, for $a_i=1$ if and only if $i\in I$, we have 
 \begin{displaymath}
\mu _n(a_1,\ldots ,a_{n}) =1.
 \end{displaymath} 
Now, if 
 \begin{displaymath}
T_{m} (p_1(a_1,\ldots ,a_{n+2}),\ldots ,p_{m+2}(a_1,\ldots ,a_{n+2}))=1
 \end{displaymath} 
with $a_i=1$ if and only if $i\in I\cup \{n+1\}$, 
then for $a_i=1$ if and only if $i\in I\cup \{n+2\}$, 
 \begin{displaymath}
 \begin{array}{lll}  
T_{n} (a_1,\ldots ,a_{n+2}) =\mu _n (a_1,\ldots ,a_{n}) =1\textrm{ and }\\
T_{m} (p_1(a_1,\ldots ,a_{n+2}),\ldots ,p_{m+2}(a_1,\ldots ,a_{n+2}))=\\
H^d_m (p_1(a_1,\ldots ,a_{n+2}),\ldots ,p_{m}(a_1,\ldots ,a_{n+2}))=0.
 \end{array}
 \end{displaymath}
Otherwise, for $a_i=1$ if and only if $i\in I\cup \{n+1\}$, 
 \begin{displaymath}
 \begin{array}{lll} 
T_{n} (a_1,\ldots ,a_{n+2}) =\mu _n (a_1,\ldots ,a_{n}) =1\textrm{ and }\\
T_{m} (p_1(a_1,\ldots ,a_{n+2}),\ldots ,p_{m+2}(a_1,\ldots ,a_{n+2}))=\\
\mu _m (p_1(a_1,\ldots ,a_{n+2}),\ldots ,p_{m}(a_1,\ldots ,a_{n+2}))=0.
 \end{array}
 \end{displaymath}

Finally, we consider the case $p_{m+1}\not=p_{m+2}$ and  $p_{m+1},p_{m+2}\in \{x_{n+1},x_{n+2}\}$.
Without loss of generality, suppose that 
$p_{m+1}=x_{n+1}$ and $ p_{m+2}=x_{n+2}$.

Note that there must be at least one $1\leq i_1\leq m$, such that 
 $p_{i_1}=x_{n+1}$ or $p_{i_1}=x_{n+2}$, otherwise, by identifying $x_{n+1}=x_{n+2}=1$
we would conclude that $H_n\preceq _{\bf V}H_m$, which contradicts Lemma~\ref{lemma5},
 or alternatively, by identifying $x_{n+1}=x_{n+2}=0$
we would conclude that $H^d_n\preceq _{\bf V}H^d_m$ which, together with Fact~\ref{fact:dual},
again constitutes a contradiction.

If there are $1\leq i_1<i_2\leq m$, such that $p_{i_1},p_{i_2}  \in \{x_{n+1},x_{n+2}\}$, then
 for $a_i=1$ if and only if $i=n+1,n+2$, 
 \begin{displaymath}
 \begin{array}{lll}  
T_{n} (a_1,\ldots ,a_{n+2}) =H_n (a_1,\ldots ,a_{n}) =0\textrm{ and }\\
T_{m} (p_1(a_1,\ldots ,a_{n+2}),\ldots ,p_{m+2}(a_1,\ldots ,a_{n+2}))=\\
H_m (p_1(a_1,\ldots ,a_{n+2}),\ldots ,p_{m}(a_1,\ldots ,a_{n+2}))=1.
 \end{array}
 \end{displaymath}
 If there is exactly one $1\leq i_1\leq m$, such that $p_{i_1}  \in \{x_{n+1},x_{n+2}\}$, then
 for $a_i=1$ if and only if $i=j,n+1,n+2$, for a unique $1\leq j\leq n$,  
 \begin{displaymath}
 \begin{array}{lll}  
T_{n} (a_1,\ldots ,a_{n+2}) =H_n (a_1,\ldots ,a_{n}) =0\textrm{ and }\\
T_{m} (p_1(a_1,\ldots ,a_{n+2}),\ldots ,p_{m+2}(a_1,\ldots ,a_{n+2}))=\\
H_m (p_1(a_1,\ldots ,a_{n+2}),\ldots ,p_{m}(a_1,\ldots ,a_{n+2}))=1.
 \end{array}
 \end{displaymath}

In all possible cases, we derive the same contradiction $T_n\not=T_m(p_1,\ldots ,p_{m+2})$, and
hence, ${T_n}\not\preceq _{\bf V}{ T_m}$.

The proof of $s_n\not\preceq _{\bf V}s_m$ can be obtained by minor adjustments in the proof above.
\end{proof}

\subsection{Classification of the closed
 intervals of ${\bf E}_{\mathbb{B}}$}

In this subsection we provide a complete classification of the closed
 intervals of ${\bf E}_{\mathbb{B}}$ in terms of their size. We prove the following theorem:

\begin{theorem}
\label{theorem:classification}
Let $\interval{\mathcal{C}_1}{\mathcal{C}_2}$ be a non-empty closed interval of ${\bf E}_{\mathbb{B}}$.
Then $\interval{\mathcal{C}_1}{\mathcal{C}_2}$ is countable 
 if and only if one of the following holds:
 \begin{itemize}
\item ${\mathcal{C}_2}\subseteq V$, 
\item ${\mathcal{C}_2}\subseteq \Lambda $, 
\item ${\mathcal{C}_2}\subseteq L$, 
\item  $\mathcal{C}\cap M_c\subseteq \mathcal{C}_1$ and $\mathcal{C}_2 \subseteq \mathcal{C}\cap M$
where $\mathcal{C}$ is a clone in 
\begin{displaymath}
\{U_m:m\geq 2\}\cup \{W_m:m\geq 2\}  \cup \{\Omega ,U_\infty , W_\infty \}.
\end{displaymath}
\end{itemize}    
\end{theorem}

The proof of Theorem~\ref{theorem:classification} follows from several propositions.

\begin{proposition}
Let $\interval{\mathcal{C}_1}{\mathcal{C}_2}$ be a closed interval of ${\bf E}_{\mathbb{B}}$.
If ${\mathcal{C}_2}\subseteq \Lambda \cup V\cup L$, then $\interval{\mathcal{C}_1}{\mathcal{C}_2}$ is 
countable.
\end{proposition}
 \begin{proof}
Using the description of the clones $\Lambda $, $V$ and $L$, it is easy to verify that every antichain
 in $\Lambda $, in $V$ or in $L$ is finite. The proof of the proposition follows then from
Theorem~\ref{theorem:characterization}.  
 \end{proof}

From the fact that $M\setminus M_c= \{{\bf 0}, {\bf 1}\}I_c$, it follows that:
\begin{proposition}
If $\mathcal{C}\in \{U_m:m\geq 2\}\cup \{W_m:m\geq 2\}  \cup \{\Omega ,U_\infty , W_\infty \}$, 
then the closed interval $\interval{\mathcal{C}\cap M_c}{\mathcal{C}\cap M}$ of 
${\bf E}_{\mathbb{B}}$ is finite.
\end{proposition}

Thus if $\interval{\mathcal{C}_1}{\mathcal{C}_2}$ satisfies the conditions of 
Theorem~\ref{theorem:classification}, then it is countable.
 To prove that these are indeed
the only countable closed intervals of ${\bf E}_{\mathbb{B}}$, we show that the minimal intervals which do
not satisfy the conditions of 
Theorem~\ref{theorem:classification} are uncountable
by making use of the antichains 
provided in Subsection 4.2 and applying Theorem~\ref{theorem:characterization}.
 This suffices to complete the proof 
of Theorem~\ref{theorem:characterization} because if a closed interval does not satisfy the conditions of 
Theorem~\ref{theorem:classification}, then it must contain a minimal interval not satisfying 
the same conditions.

\begin{proposition}\label{proposition1}
Each of the minimal intervals 
\begin{itemize}
\item[$(i)$] $\interval{\mathcal{C}}{\Omega }$ where $\mathcal{C}\in \{T_0,T_1,L,S,M\}$,
\item[$(ii)$] $\interval{\mathcal{C}}{T_0}$ where $\mathcal{C}\in \{T_c,L_0,M_0,U_2\}$,
\item[$(iii)$] $\interval{\mathcal{C}}{T_1}$ where $\mathcal{C}\in \{T_c,L_1,M_1,W_2\}$,
\item[$(iv)$]  $\interval{\mathcal{C}}{T_c}$ where $\mathcal{C}\in \{M_c,S_c,T_cU_2,T_cW_2\}$,
\end{itemize}
is uncountable.
\end{proposition}

\begin{proof} Note that every member of $(f_n)_{n\geq 4}$ defined in Lemma~\ref{lemma1} 
belongs to $T_0 \setminus T_1\cup T_c\cup U_2\cup S\cup M$. Moreover, if $n\geq 5$, then $f_n\not\in L$.
Thus, using Fact~\ref{fact:dual} and applying Theorem~\ref{theorem:characterization}, we conclude that
$(i)$, $(ii)$ and $(iii)$ of the proposition hold. 
The proof of $(iv)$ follows similarly by observing that every member of $(g_n)_{n\geq 4}$ defined
 in Lemma~\ref{lemma2} 
belongs to $T_c \setminus (M_c\cup S_c\cup T_cU_2\cup T_cW_2)$. 
\end{proof}

\begin{proposition}\label{proposition2}
Each of the minimal intervals 
\begin{itemize}
\item[$(i)$] $\interval{ \mathcal{C}_1\cap \mathcal{C}_2}{\mathcal{C}_2 }$ 
where $\mathcal{C}_1\in \{T_c,M\}$ and 
$\mathcal{C}_2\in \{U_m:m\geq 2\} \cup \{U_\infty \}$,
\item[$(ii)$] $\interval{ \mathcal{C}_1\cap \mathcal{C}_2}{\mathcal{C}_2 }$ 
where $\mathcal{C}_1\in \{T_c,M\}$ and 
$\mathcal{C}_2\in \{W_m:m\geq 2\} \cup \{W_\infty \}$,
\item[$(iii)$] $\interval{ M_c\cap \mathcal{C}_2}{\mathcal{C}_2 }$ 
where $\mathcal{C}_2\in \{T_cU_m:m\geq 2\} \cup \{T_cU_\infty \}$,
\item[$(iv)$] $\interval{ M_c\cap \mathcal{C}_2}{\mathcal{C}_2 }$ 
where $\mathcal{C}_2\in \{T_cW_m:m\geq 2\} \cup \{T_cW_\infty \}$,
\end{itemize}
is uncountable.
\end{proposition}

\begin{proof} Observe that, for every $n\geq 4$, $u_n\in U_\infty \setminus T_cU_2$, 
and $t^u_n\in T_cU_\infty \setminus MU_2$. Thus it follows from Lemma~\ref{lemma3} and   
Theorem~\ref{theorem:characterization} that $(i)$ and $(iii)$ hold. The proof of $(ii)$ and $(iv)$ 
follows similarly by making use of Fact~\ref{fact:dual}.  
\end{proof}
  
\begin{proposition}\label{proposition3}
Each of the minimal intervals 
\begin{itemize}
\item[$(i)$] $\interval{ M\cap \mathcal{C}}{M}$ 
where $\mathcal{C}\in \{\Lambda ,V\}$,
\item[$(ii)$] $\interval{ MU_2}{M_0}$ and $\interval{ MW_2}{M_1}$, 
\item[$(iii)$] $\interval{ M_c\cap \mathcal{C}}{M_c}$ 
where $\mathcal{C}\in \{U_2,W_2\}$,
\item[$(iv)$] $\interval{ SM}{\mathcal{C}}$ 
where $\mathcal{C}\in \{M_cU_2,M_cW_2\}$,
\end{itemize}
is uncountable.
\end{proposition}

\begin{proof} Observe that for each $n\geq 4$, we have $H_n \in M_cW_2\setminus (U_2 \cup \Lambda  \cup V)$, 
and thus, by Lemma~\ref{lemma5} and Fact~\ref{fact:dual},
 $(H_n)_{n\geq 4}$ and $(H^d_n)_{n\geq 4}$ constitute infinite antichains in  
$M_cW_2\setminus (U_2 \cup \Lambda  \cup V)$ and $M_cU_2\setminus (W_2 \cup \Lambda  \cup V)$, respectively.
Hence, by Theorem~\ref{theorem:characterization} the proposition holds.
\end{proof}

\begin{proposition}\label{proposition4}
For $n\geq 2$, each of the minimal intervals 
\begin{itemize}
\item[$(i)$] $\interval{ \mathcal{C}\cap U_{n+1}}{\mathcal{C}\cap U_{n}}$ 
where $\mathcal{C}\in \{\Omega ,T_c,M,M_c\}$,
\item[$(ii)$] $\interval{ \mathcal{C}\cap W_{n+1}}{\mathcal{C}\cap W_{n}}$ 
where $\mathcal{C}\in \{\Omega ,T_c,M,M_c\}$,
\end{itemize}
is uncountable.
\end{proposition}

\begin{proof} It is not difficult to verify that 
for each $n\geq 2$, $H_{n+1}\in M_cW_{n}\setminus W_{n+1}$ and 
thus, by Lemma~\ref{lemma5}, 
$(H_m)_{m\geq n+1}$ constitutes an infinite antichain in 
$M_cW_{n}\setminus W_\infty $.  
Furthermore, for each $n\geq 2$, the family $(G^{n+1}_m)_{m\geq n+1}$
is in $M_cW_{n}$ but for every $m\geq n+1$, $G^{n+1}_m\not\in W_{n+1} $, otherwise by identifying the variables $x_{n+1},\ldots ,x_{m+n}$
of $G^{n+1}_m$ we would conclude that $H_{n+1}\in W_{n+1} $ which is a contradiction.
By Lemma~\ref{lemma5}, for $n\geq 2$, $(G^{n+1}_m)_{m\geq n+1}$ constitutes an infinite antichain 
in $M_cW_{n}\setminus W_{n+1}$. By this fact, we have that $(ii)$ holds.
The proof $(i)$ of the proposition follows now by making use of Fact~\ref{fact:dual} and applying Theorem~\ref{theorem:characterization}. 
\end{proof}
  
\begin{proposition}\label{proposition5}
Each of the minimal intervals 
\begin{itemize}
\item[$(i)$] $\interval{\mathcal{C}\cap \Lambda }{\mathcal{C}\cap U_\infty }$ 
where $\mathcal{C}\in \{M_0,M_c\}$,
\item[$(ii)$] $\interval{\mathcal{C}\cap V }{\mathcal{C}\cap W_\infty }$ 
where $\mathcal{C}\in \{M_1,M_c\}$,
\end{itemize}
is uncountable.
\end{proposition}

\begin{proof} Observe that each member of $(G^2_m)_{m\geq 2}$ is in $M_cW_\infty \setminus V $, 
and thus, by Lemma~\ref{lemma5},
 $(G^2_m)_{m\geq 2}$ constitutes an infinite antichain in  
$M_cW_\infty \setminus  V $.
The proof of the proposition follows now by making use of Fact~\ref{fact:dual}
 and applying Theorem~\ref{theorem:characterization}. 
\end{proof}

\begin{proposition}\label{proposition6}
Each of the minimal intervals 
\begin{itemize}
\item[$(i)$] $\interval{I_c}{SM}$,
\item[$(ii)$] $\interval{\mathcal{C}\cap S_c }{S_c}$ 
where $\mathcal{C}\in \{M,L\}$,
\item[$(iii)$] $\interval{ \mathcal{C}\cap S }{S}$ 
where $\mathcal{C}\in \{T_c,L\}$,
\end{itemize}
is uncountable.
\end{proposition}

\begin{proof} To prove Proposition~\ref{proposition6} we shall make use of the antichains given in 
Lemma~\ref{lemma7}.
First, we show that the members of $(T_n)_{n\in \mathbb{O}}$ are in $SM\setminus L$.
Observe that if 
 \begin{displaymath} 
T_n(x_1,\ldots ,x_n,x_{n+1}, x_{n+2})=H_n(x_1,\ldots ,x_n), 
  \end{displaymath} 
then $x_{n+1}=x_{n+2}=1$, i.e. ${\bar x_{n+1}}={\bar x_{n+2}}=0$. Hence, 
\begin{displaymath}
 \begin{array}{llll} 
T^d_n(x_1,\ldots ,x_n,x_{n+1}, x_{n+2})={\bar H^d_n}({\bar x_1},\ldots ,{\bar x_n})=\\
H_n(x_1,\ldots ,x_n)=T_n(x_1,\ldots ,x_n,x_{n+1}, x_{n+2})
 \end{array}
\end{displaymath} 
For the case $T_n(x_1,\ldots ,x_n,x_{n+1}, x_{n+2})=\mu _n(x_1,\ldots ,x_n)$, we note that the identity  
$x_{n+1}+x_{n+2}={\bar x_{n+1}}+{\bar x_{n+2}}$ holds, and that $\mu _n$ is self-dual, and thus 
$T^d_n=T_n$ also holds. Hence, $(T_n)_{n\in \mathbb{O}}$ is a family of self-dual functions.
The fact that each function in $(T_n)_{n\in \mathbb{O}}$ is monotone follows immediatly from the definition
of each $T_n$.

To see that the members of $(T_n)_{n\in \mathbb{O}}$ are not linear, just note that
for each $n\in \mathbb{O}$, 
 \begin{displaymath} 
 \begin{array}{ll}  
T_n(1,\ldots ,1,x_{n+1}, x_{n+2})= T_n(1,\ldots ,1,x_{n+1}, {\bar x_{n+2}})=\\
T_n(1,\ldots ,1,{\bar x_{n+1}}, {\bar x_{n+2}}) \textrm{ for all $x_{n+1}, x_{n+2}\in \{0,1\} $,} 
 \end{array}
  \end{displaymath} 
and $T_n$ depends essentially on all variables.

Using Fact~\ref{fact:dual}, it follows from Lemma~\ref{lemma7} that $(T_n)_{n\in \mathbb{O}}$
and $({\bar T_n})_{n\in \mathbb{O}}$ are antichains in $SM\setminus L$ and $S\setminus (S_c\cup L)$, 
respectively. Thus, by Theorem~\ref{theorem:characterization} 
it follows that $(i)$ and $(iii)$ of the proposition hold. 

Now we show that the members of $(s_n)_{n\in \mathbb{O}}$ are in $S_c\setminus (SM\cup L)$. 
It is easy to verify that indeed, for each $n\in \mathbb{O}$, $s_n\in S_c\setminus L$.
To see that $s_n\in S_c\setminus M$, let $n\in \mathbb{O}$ and 
consider the $n+2$-tuples $ {\bf a} =(1,1,0\ldots ,0,0,0)$ and
${\bf b}=(1,1,0,\ldots ,0,1, 1)$. 
Obviously, ${\bf a}\preceq {\bf b}$ but $s_n({\bf a})> s_n({\bf b})$.
Thus, by Lemma~\ref{lemma7}, $(s_n)_{n\in \mathbb{O}}$ constitutes an infinite antichain in 
$S_c\setminus (SM\cup L)$, and hence, by Theorem~\ref{theorem:characterization} 
it follows that $(ii)$ also holds. 
\end{proof}
    
\subsection{Characterization of the closed
 intervals of ${\bf E}_{\mathbb{B}}$}
Using the classification of the closed intervals of ${\bf E}_{\mathbb{B}}$ given in the previous 
subsection, we derive the following characterization of
 the uncountable closed intervals of ${\bf E}_{\mathbb{B}}$: 

\begin{theorem}
\label{theorem:characterization2}
Let $\interval{\mathcal{C}_1}{\mathcal{C}_2}$ be a closed interval of ${\bf E}_{\mathbb{B}}$.
Then there are uncountably many equational classes in $\interval{\mathcal{C}_1}{\mathcal{C}_2}$ 
 if and only if $\mathcal{C}_2\setminus \mathcal{C}_1$ contains a non-associative Boolean function.    
\end{theorem}

\begin{proof} Let $\interval{\mathcal{C}_1}{\mathcal{C}_2}$ be a closed interval of ${\bf E}_{\mathbb{B}}$.
It is not difficult to verify that if $\interval{\mathcal{C}_1}{\mathcal{C}_2}$ satisfies one of 
the conditions of Theorem~\ref{theorem:classification}, then $\mathcal{C}_2\setminus \mathcal{C}_1$ 
contains only quasi-associative Boolean functions. 

To see that the converse holds, it is enough to provide a non-associative function in 
$\mathcal{C}_2\setminus \mathcal{C}_1$ 
for each uncountable minimal interval $\interval{\mathcal{C}_1}{\mathcal{C}_2}$.
 For that it is sufficient to show that the members of the
 antichains given in Subsection 4.2 are non-associative. Note that the members of each antichain $(F_n)_{n\in I}$
 given in Subsection 4.2 have no
inessential variables, and thus by Proposition~\ref{quasi}, for each antichain $(F_n)_{n\in I}$,
it is enough to show that for some $n\in I$, and some 
$1\leq i<j\leq m_n$, there is $(a_1,\ldots ,a_{2m_n-1})\in \mathbb{B}^{2m_n-1}$, such that  
 \begin{displaymath}
\begin{array}{l}
F_n(a_1,\ldots ,a_{i-1},F_n(a_i,\ldots ,a_{i+m_n-1}),a_{i+m_n},\ldots ,a_{ 2m_n-1})\not=
\\ F_n(a_1,\ldots ,a_{j-1},F_n(a_j,\ldots ,a_{j+m_n-1}),a_{j+m_n},\ldots ,a_{ 2m_n-1})
\end{array}
\end{displaymath}

Consider the antichain $(f_n)_{n\geq 4}$ defined in Lemma~\ref{lemma1}. 
Let $n>4$. To see that $f_n$ is non-associative, let $i=2$, $j=3$, and let 
$(a_1,\ldots ,a_{2n-1})\in \mathbb{B}^{2n-1}$ be defined by $a_t=0$ if and only if 
$t\in \{1,\ldots ,n-1,n+1\}$.
Then 
 \begin{displaymath}
\begin{array}{l}
f_n(a_1,f_n(a_2\ldots ,a_{n+1}),a_{n+2},\ldots ,a_{ 2n-1})=1\not=
\\ 0=f_n(a_1,a_2, f_n(a_3,\ldots ,a_{n+2}),a_{n+3},\ldots ,a_{ 2n-1})
\end{array}
\end{displaymath}

For the antichain $(g_n)_{n\geq 4}$ defined in Lemma~\ref{lemma2}, 
let $n>4$, $i=1$, $j=3$, and let 
$(a_1,\ldots ,a_{2n-1})\in \mathbb{B}^{2n-1}$ be defined by $a_t=1$ if and only if 
$t\in \{1,n+1,\ldots ,2n-1\}$. Then  
 \begin{displaymath}
\begin{array}{l}
g_n(g_n(a_1,\ldots ,a_{n}),a_{n+1},\ldots ,a_{ 2n-1})=1\not=
\\ 0=g_n(a_1,a_2,g_n(a_3,\ldots ,a_{n+2}),a_{n+3},\ldots ,a_{ 2n-1})
\end{array}
\end{displaymath}

For the antichain $(u_n)_{n\geq 4}$ defined in Lemma~\ref{lemma3}, 
let $n>4$, $i=1$, $j=2$, and let 
$(a_1,\ldots ,a_{2n+1})\in \mathbb{B}^{2n+1}$ be defined by $a_t=0$ if and only if 
$t=2$. Then  
 \begin{displaymath}
\begin{array}{l}
u_n(u_n(a_1,\ldots ,a_{n+1}),a_{n+2},\ldots ,a_{ 2n+1})=1\wedge 0=0\not=
\\ 1=1\wedge 1=u_n(a_1,u_n(a_2,\ldots ,a_{n+2}),a_{n+3},\ldots ,a_{ 2n+1})
\end{array}
\end{displaymath}
For the antichain $(t^u_n)_{n\geq 4}$ also defined in Lemma~\ref{lemma3}, 
let $n>4$, $i=1$, $j=2$, and let 
$(a_1,\ldots ,a_{2n+1})\in \mathbb{B}^{2n+1}$ be defined by $a_t=0$ if and only if 
$2\leq t\leq n-1$. Then  
 \begin{displaymath}
\begin{array}{l}
t^u_n(t^u_n(a_1,\ldots ,a_{n+1}),a_{n+2},\ldots ,a_{ 2n+1})=1\wedge 1=1\not=
\\ 0=1\wedge 0=t^u_n(a_1,t^u_n(a_2,\ldots ,a_{n+2}),a_{n+3},\ldots ,a_{ 2n+1})
\end{array}
\end{displaymath}

For the antichain $(H_n)_{n\geq 4}$ defined in Lemma~\ref{lemma5}, 
let $n>4$, $i=1$, $j=2$, and let 
$(a_1,\ldots ,a_{2n-1})\in \mathbb{B}^{2n-1}$ be defined by $a_t=1$ if and only if 
$1\leq t\leq n$. Then  
 \begin{displaymath}
\begin{array}{l}
H_n(H_n(a_1,\ldots ,a_n),a_{n+1},\ldots ,a_{ 2n-1})=0\not=
\\ 1=H_n(a_1,H_n(a_2,\ldots ,a_{n+1}),a_{n+2},\ldots ,a_{ 2n-1})
\end{array}
\end{displaymath}
For the antichain $(G_m^n)_{m\geq n}$ also defined in Lemma~\ref{lemma5}, 
let $m\geq n\geq 4$, $i=2$, $j=3$, and let 
$(a_1,\ldots ,a_{2m+2n-3})\in \mathbb{B}^{2m+2n-3}$ be defined by $a_t=1$ if and only if 
$1\leq t\leq 2$. Then  
 \begin{displaymath}
\begin{array}{l}
G_m^n(a_1,G_m^n(a_2,\ldots ,a_{m+n}),a_{m+n+1},\ldots ,a_{2m+ 2n-3})=0\not=
\\ 1=G_m^n(a_1,a_2,G_m^n(a_3,\ldots ,a_{m+n+1}),a_{m+n+2},\ldots ,a_{2m+ 2n-3})
\end{array}
\end{displaymath}

For the antichain $(T_n)_{n\in \mathbb{O}}$ defined in Lemma~\ref{lemma7}, 
let $n\in \mathbb{O}$, $i=1$, $j=2$, and let 
$(a_1,\ldots ,a_{2n+3})\in \mathbb{B}^{2n+3}$ be defined by $a_t=1$ if and only if 
$t\in \{1,\ldots ,n+2,2n+2 ,2n+3\}$. Then  
 \begin{displaymath}
\begin{array}{l}
T_n(T_n(a_1,\ldots ,a_{n+2}),a_{n+3},\ldots ,a_{ 2n+3})=0\not=
\\ 1=T_n(a_1,T_n(a_2,\ldots ,a_{n+3}),a_{n+4},\ldots ,a_{ 2n+3})
\end{array}
\end{displaymath}
For the antichain $(s_n)_{n\in \mathbb{O}}$ also defined in Lemma~\ref{lemma7}, 
let $n\in \mathbb{O}$ , $i=1$, $j=2$, and let 
$(a_1,\ldots ,a_{2n+3})\in \mathbb{B}^{2n+3}$ be defined by $a_t=1$ if and only if 
$t\in \{1,\ldots ,n+2\}$. Then  
 \begin{displaymath}
\begin{array}{l}
s_n(s_n(a_1,\ldots ,a_{n+2}),a_{n+3},\ldots ,a_{ 2n+3})=0\not=
\\ 1=s_n(a_1,s_n(a_2,\ldots ,a_{n+3}),a_{n+4},\ldots ,a_{ 2n+3})
\end{array}
\end{displaymath}

Note that if a Boolean function $f$ is non-associative, then its dual is also non-associative.
Now, if $\interval{\mathcal{C}_1}{\mathcal{C}_2}$ is a minimal and uncountable closed interval, then 
$\mathcal{C}_2\setminus \mathcal{C}_1$ contains at least one of the functions above or the dual of 
one of the functions above, and thus it contains
a non-associative function. Since each uncountable closed interval must contain a minimal
 and uncountable closed interval, we conclude that if $\interval{\mathcal{C}_1}{\mathcal{C}_2}$ 
 is an uncountable closed interval, then $\mathcal{C}_2\setminus \mathcal{C}_1$ 
contains a non-associative Boolean function, and the proof of the theorem is complete.  
\end{proof}

\end{document}